\documentclass[smallextended]{svjour3}   
\usepackage{microtype}
\usepackage{algorithmic}
\usepackage[ruled,vlined]{algorithm2e}
\usepackage{graphicx,array}
\usepackage{subfigure,tabularx}
\usepackage{booktabs} 
\usepackage{amssymb,amsmath,amsfonts,mathrsfs}
\usepackage{etoolbox,verbatim,color}
\usepackage{url}
\usepackage{hyperref}
\usepackage{xcolor}  
\usepackage{adjustbox,caption}
\newcommand{\revise}[1]{{{\color{black} #1}}}
\newcommand{\secondrevise}[1]{{{\color{black} #1}}}

\newcommand{\mbfl}{\mathbf l}
\newcommand{\mL}{\mathcal L} 
 
\newcommand{\mB}{\mathcal B}

\newcommand{\mG}{\mathcal G}

\newcommand{\dom}{{\rm{dom}\,}}

\newcommand{\mbbR}{\mathbb R} 

\DeclareMathOperator*{\argmin}{argmin} 
\DeclareMathOperator*{\minimize}{minimize} 

\newtheorem{assumption}{Assumption}

\begin{document}

\title{An inertial ADMM for a class of nonconvex composite optimization with nonlinear coupling constraints} 
\titlerunning{ADMM for nonconvex composite optimization}    
\author{Le Thi Khanh Hien         \and 
        Dimitri Papadimitriou 
}

\authorrunning{L. T. K. Hien, D. Papadimitriou} 

\institute{L. T. K. Hien finished this work when she was at Huawei Belgium Research Center,  
Leuven 3001, Belgium \\
D. Papadimitriou \at  FTNO,  Huawei Belgium Research Center,  
Leuven 3001, Belgium\\
              \email{khanhhiennt@gmail.com, dimitrios.papadimitriou.ext@huawei.com}  
}

\maketitle

\begin{abstract}
In this paper, we propose an inertial alternating direction method of multipliers for solving a class of non-convex multi-block optimization problems with \emph{nonlinear coupling constraints}.  Distinctive features of our proposed method, when compared with other alternating direction methods of multipliers for solving non-convex problems with nonlinear coupling constraints, include: (i) we apply the inertial technique to the update of primal variables and (ii) we apply a non-standard update rule for the multiplier by scaling the multiplier by a factor  before moving along the ascent direction where a relaxation parameter is allowed. Subsequential convergence and global convergence are presented for the proposed algorithm. 
\end{abstract}



\section{Introduction}
\label{sec:Intro}

\revise{
We consider the following composite problem with nonlinear coupling constraints 
\begin{equation}
\label{CP-main}
\begin{split}
\minimize_{x=(x_1,\ldots,x_s)\in\mbbR^n,y\in\mbbR^m} & \quad \Theta(x,y):= F(x) + \sum_{i=1}^s f_i(x_i)   + G(y) \\
\mbox{subject to} &\quad h(x)+ \mB y=0,
\end{split}
\end{equation}
where  $h(x)=(h_1(x),\ldots,h_q(x))$ is a mapping from $\mbbR^n$ to $\mbbR^q$, $h_i(x)$ for $i=1,\ldots,q$, are continuously differentiable functions, $f_i:\mbbR^n\to \mbbR\cup \{+\infty\}$ are proper lower semi-continuous functions, $F:\mbbR^n\to \mbbR$ is a  continuously differentiable function, $G$ is $L_G$-smooth (that is, $\nabla G$ is $L_G$-Lipschitz continuous), and $\mB$ is a linear mapping from $\mbbR^m$ to $\mbbR^q$. The following Assumption \ref{assump:h2} is necessary for our convergence analysis. 
 In Remark \ref{remark_1}, we will also consider a situation when Assumption~\ref{assump:h2}(i) can be removed. Note that Assumption~\ref{assump:h2}(ii) is typical in the literature of ADMM for nonconvex problem.
\begin{assumption}
\begin{itemize}
\item[(i)] \label{assump:h2} 
$x_i\mapsto h_2(x):=\frac12\|h(x)\|^2$ is $l_i(x_1,\ldots,x_{i-1},x_{i+1},\ldots,x_s)$-smooth, for $i= 1,\ldots,s$; that is, $\nabla_{x_i} h_2(x)$ is $l_i$-Lipschiz continuous.
\item[(ii)]  \label{assump:theta} 
 $\sigma_{\mB}:=\lambda_{\min}(\mB \mB^\top) >0$ ($\sigma_{\mB}$  denotes the smallest eigenvalue of $\mB \mB^\top$)  and $\Theta(x,y)$ is bounded from below by $\nu$. 
\end{itemize}
\end{assumption} }
  Let us provide some examples that satisfy Assumption \ref{assump:h2}(i).
 \begin{itemize} 
 \item 
 If $h(x)=\sum_{i=1}^s \mathcal A_i x_i$ ( $h$ is linear), then $x_i\mapsto \frac12\|h(x_i,y_{\ne i})\|^2$ is $\|\mathcal A_i^{\revise{\top}} \mathcal A_i\|$-smooth. 
 \item If $h(x_1,x_2)=x_1 x_2$ (bi-linear function), then $x_1\mapsto 
 \frac12\|h(x_1,x_2)\|^2$ is $\|x_2 x_2^\top\|$-smooth and $x_2\mapsto \frac12\|h(x_1,x_2)\|^2$ is $\|x_1^\top x_1\|$-smooth.
 \item If $h$ is a multilinear mapping then $h$ satisfies Assumption \ref{assump:h2}(i).  
 \end{itemize}
%

 Alternating direction methods of multipliers (ADMM) for solving Problem \eqref{CP-main} with \emph{linear} coupling constraints (that is when \revise{$h(x)$ is an affine mapping}) have been deeply studied in the literature, see e.g., \cite{Bot2020,Hien_iADMM,HongADMM2020,Wang2019} and the references therein. However, ADMM \revise{with convergence guarantee\footnote{We thank the reviewer for bringing to our attention the reference \cite{WANG20211_admm}, which explores the application of ADMM for addressing a nonconvex optimization problem involving nonlinear coupling constraints.  Nevertheless, it is noteworthy that the paper does not provide a convergence analysis. }} for solving \revise{nonconvex composite problem} with \emph{nonlinear} coupling constraints have only appeared in \cite{Bolte2018_ADMM}, \cite{Cohen2021} and \cite{mADMM2022}.  The authors in  \cite{Bolte2018_ADMM} consider \revise{Problem \eqref{CP-main} with a more general nonlinear coupling constraint, which replaces $\mathcal B y$ by a mapping $g(y)$ from $\mbbR^m$ to $\mbbR^q$}, and propose a universal framework to study global convergence analysis of \emph{Lagrangian sequences} (see \cite[Section 3.3]{Bolte2018_ADMM}). To establish the global convergence of the Lagrangian sequences, the notion of \emph{information zone} was introduced in \cite{Bolte2018_ADMM} and the boundedness of the multiplier sequence is the key assumption to fulfill the role of the information zone. The authors in \cite{mADMM2022}  develop the idea of 
 information zone of \cite{Bolte2018_ADMM} to propose mADMM - a multiblock alternating direction method of multipliers  - for solving the general problem with any $s\geq 1$. Although the convergence analysis presented in \cite{mADMM2022} does not use the property of the Lagrangian sequence proposed in \cite{Bolte2018_ADMM} to establish the global convergence of its generated sequence, it still relies on the boundedness assumption of the multiplier sequence (as a consequence of using the information zone). 
  \revise{To avoid this unrealistic assumption, the authors in \cite{Cohen2021} consider Problem \eqref{CP-main} (instead of the general nonlinear coupling constraints as in \cite{Bolte2018_ADMM}) with $s=1$.} They design a proximal linearized alternating direction method of multipliers that requires  a backtracking procedure to generate the proximal parameters. The convergence analysis of \cite{Cohen2021} does not require the boundedness of the multiplier sequence but the backtracking procedure used in \cite{Cohen2021} relies on the boundedness assumption of the generated sequence (see, \cite[Lemma 5.2]{Cohen2021}) to guarantee the boundedness of the proximal parameters; however, \cite{Cohen2021} does not present a  sufficient condition to guarantee the boundedness of the generated sequence.

\secondrevise{Utilizing inertial techniques has become a prevalent approach to numerically accelerate an algorithm, frequently resulting in improved convergence rates. } Let us provide a very brief review for inertial techniques. As far as we know,  \cite{POLYAK1964} is the first paper using an inertial technique that adds an inertial force, also known as a ``momentum", $\zeta^k(x^k-x^{k-1})$ (where $x^k$ is the current iterate, $x^{k-1}$ is the previous iterate, and $\zeta^k$ is an extrapolation parameter) 
 to the gradient direction to accelerate the gradient descent method. Later, Nesterov proposed his well-known 
accelerated fast gradient methods in the series of works  \cite{Nesterov1983,Nesterov1998,Nesterov2004,Nesterov2005}. Since the appearance of Nesterov's accelerated gradient methods, inertial techniques have been widely applied in convex as well as non-convex problems to accelerate the convergence of first-order methods, see e.g.,  \cite{Hien_iADMM,Hien_ICML2020,Titan2020,Ochs2019,Ochs2014,Pock2016,Xu2013,Xu2017,Zavriev1993} and the references therein.

In this paper, we employ the inertial technique proposed in \cite{Titan2020} to propose iADMMn - an inertial alternating direction method of multipliers for solving  Problem \eqref{CP-main}. 
 When $\mB  =-\mathcal I$, where $\mathcal I$ is an identity mapping, Problem \eqref{CP-main} is equivalent to
\begin{equation}
 \label{regularized_nlq}
 \minimize_{x\in \mathbb R^n}  F(x)  + \sum_{i=1}^s f_i(x_i) + G(h(x)).
 \end{equation}
 An example of \eqref{regularized_nlq} is the following logistic matrix factorization problem  \cite{Liu2016_PLOS}
\begin{equation}
    \label{log_MF}
    \min_{U,V} \sum_{i=1}^m \sum_{j=1}^n (1+c y_{ij}-y_{ij}) \log (1+\exp(u_i v_j^\top))-c y_{ij}u_i v_j^\top + \frac{\lambda_d}{2} \|U\|_F^2 + \frac{\lambda_t}{2}\|V\|_F^2, 
\end{equation}
where $Y=\mbbR^{m\times n}$ is a given data set with each element $y_{ij}\in \{0,1\}$, $\lambda_d$ and $\lambda_t$ are regularization parameters, 
and $c$ is some given constant. 
Problem \eqref{log_MF} has the form of Problem \eqref{regularized_nlq} with  $x=(U,V)$, $F(x)=\frac{\lambda_d}{2} \|U\|_F^2 + \frac{\lambda_t}{2}\|V\|_F^2$, $f_i=0$, $h(U,V)=UV$,  and $G(W)=\sum_{i,j} (1+c y_{ij}-y_{ij})  \log (1+\exp(W_{ij}))-c y_{ij} W_{ij}$. 
A few other examples of Problem~\eqref{regularized_nlq} are the PDE-constrained inverse problem  \cite{Roosta2014},  the risk parity portfolio selection problem \cite{Maillard2010}, the robust phase retrieval problem \cite{Duchi2018}, 
 the nonlinear regression problem \cite{Dutter1981,GoodBengCour16} ($h$ represents the model to train, $G$ is a loss function, and $ F(x)  + \sum_{i=1}^s f_i(x_i)$ is a regularizer), and the generative adversarial networks  \cite{mao2017squares}.
 
  It is worth noting that  inertial techniques have also been applied in \cite{Hien_iADMM} to accelerate the convergence of ADMM for solving Problem  \eqref{CP-main} with \revise{$h$ being a linear mapping}. We allow $h$ to be nonlinear in this paper.  On the other hand, iADMMn uses a non-standard update rule for the multiplier. More specifically, the multiplier is scaled by a factor before moving along the ascent direction, see \eqref{eq:omega_update}. This update rule for the multipliers is inspired by recent papers \cite{SunSun2021,YANG2022110551} which give new perspectives on the multiplier update of primal-dual methods in the non-convex setting. Specifically,  classical primal-dual methods are interpreted as methods that  alternatively update the primal variable  by minimizing the primal problem (primal descent) and update the dual (the multiplier)  by maximizing the dual problem (dual ascent), see \cite[Chapter 7]{Bertsekas}. However, when the primal problem is highly non-convex and may not be done in closed form, 
  the classic dual ascent step may lose its valid interpretation. Hence, it makes sense to consider some modifications for the dual update. For example, \cite{SunSun2021} proposes a scaled dual \emph{descent} update and \cite{YANG2022110551} imposes a discounting factor to the multiplier before moving along the ascent direction. The multiplier update  \eqref{eq:omega_update} of our iADMMn is similar to the multiplier update proposed in \cite{YANG2022110551} in the sense that it scales the multiplier by a factor $\tau_1$ before moving along the ascent direction, but iADMMn also allows a relaxation parameter $\tau_2$ in the step-size of updating the multiplier. 

The paper is organized as follows. In the next section, we provide some preliminaries on block surrogate functions, inertial techniques, and the augmented Lagrangian function. In Section \ref{sec:analysis}, we describe iADMMn and analyze its convergence properties. We conclude the paper in Section \ref{sec:conclusion}.

\paragraph{Notations} We denote $[s]= \{1,\ldots,s\}$. For $y\in \mbbR^s$, we use  $y_{\ne i}$ to denote $(y_1,\ldots,y_{i-1},y_{i+1},\ldots,y_s)$. For a given mapping $h:\mbbR^n\to \mbbR^s$, we use $\nabla h (x) \in \mbbR^{s\times n}$ to denote the Jacobian of $h$ at $x$, that is $ \nabla h(x)=[\nabla h_1(x) \ldots \nabla h_s(x)]^\top$, and we use $\nabla_{x_i} h(x)\in \mbbR^{s\times n_i}$ to denote the partial Jacobian of $h$ with respect to $x_i$, that is $\nabla_{x_i} h(x) = [\nabla_{x_i} h_1(x) \ldots \nabla_{x_i} h_s(x)]^\top$. For a given sequence $\{y^k\}$, we denote $\Delta y^k=y^{k}-y^{k-1}$.

\section{Preliminaries}
We refer the readers to \cite[Appendix 1]{Hien_iADMM} for some preliminaries of non-convex optimization such as the definition of limiting subdifferential and the definition of the K{\L} property. In the following, we give some preliminaries of block surrogate functions, inertial technique, augmented Lagrangian function, and $\varepsilon$-stationary point. 
\subsection{Block surrogate and nearly sufficiently decrease property}
\revise{In our upcoming analysis, we use the following definition for a surrogate function.} 
\begin{definition}[Block surrogate function]
\label{def:surrogate-block} Let $\mathcal X_i\subseteq \mathbb R^{ n_i}$, for $i=1,\ldots s$, and $\mathcal X\subseteq \mathbb R^{ n}$. 
A continuous function $u_i:\mathcal X_i \times \mathcal X \to \mathbb R  $ is called a block $x_i$ surrogate function of $\varphi:\mbbR^n\to \mbbR$ on $\mathcal X=\mathcal X_1\times\ldots\times \mathcal X_s$
if the following conditions are satisfied:
\begin{itemize}
    \item[(a)] $u_i(z_i,z) = \varphi(z)$ for all $z\in \mathcal X$, 
    \item[(b)] $u_i(x_i,z) \geq \varphi(x_i,z_{\ne i})$ for all $x_i\in\mathcal X_i$ and $z\in \mathcal X$,  where  
    $(x_i,z_{\ne i})$ denotes $(z_1,\ldots,z_{i-1},x_i,z_{i+1},\ldots,z_s),$ 
   \revise{ \item[(c)] the block approximation error is defined as $ e_i(x_i,z):=u_i(x_i,z) - \varphi(x_i,z_{\ne i})$, $x_i\mapsto e_i(x_i,z)$ is continuously differentiable and it satisfies  $\nabla_{x_i} e_i(x_i,x)=0 $ for all $x\in\mathcal X$.}
\end{itemize}

\end{definition}
\revise{For example,  when $\nabla_{x_i} \varphi(\cdot,z_{\ne i})$ is $L_i^{(z)}$-Lipschitz continuous (note that $L^{(z)}_i$ may depend on $z$), the Lipschitz gradient surrogate function~\cite{Mairal_ICML13,Xu2017,Hien_ICML2020}  has the form 
$
u_i(x_i,z) = \varphi(z) + \langle\nabla_i \varphi(z), x_i- z_i\rangle + \frac{\kappa_i L^{(z)}_i}{2}\|x_i - z_i\|^2,
$
where $\kappa_i\geq 1$.  Finding $\arg\min_{x_i} u_i(x_i,z)+ f_i(x_i)$, where $z$ plays the role of the current iterate,  would lead to the block proximal gradient step of block coordinate methods \cite{Beck2013,Bolte2014,Razaviyayn2013,Tseng2009} for solving $\min_x \varphi(x)+\sum_{i=1}^s f_i(x_i)$.} More examples such as a quadratic surrogate, a DC programming surrogate, a saddle point surrogate, and a Jensen surrogate can be found in \cite{Hien_iADMM,Titan2020,Mairal_ICML13}. \revise{Considering the optimization problem $\min_x \varphi(x)+\sum_{i=1}^s f_i(x_i)$, convergence analysis of various block coordinate methods corresponding to different choices of the surrogates can be unified by studying the convergence analysis of the block majorization minimization algorithm, which updates block $x_i$ by finding $\arg\min_{x_i} u_i(x_i,z)+ f_i(x_i)$, given $z$ being the current iterate \cite{Razaviyayn2013}. Furthermore, a suitable surrogate may lead to a closed-form solution for the update of block $x_i$ while alternatingly minimizing the objective function does not have a closed-form solution and requires an outer optimizer to solve the subproblem, see e.g., \cite[Section 6.2]{Titan2020}.}  

The authors in \cite{Titan2020} propose TITAN, an inertial block majorization minimization framework for solving non-convex composite optimization problems \revise{without coupling constraints}. TITAN updates one block of variables at a time by minimizing an inertial block surrogate function that is formed by adding an inertial force to a block surrogate function of the objective. The key property to establish the convergence of TITAN is the following nearly sufficiently decreasing property (NSDP), which holds when an inertial block surrogate is minimized (that is when  \eqref{inertial_general} is performed). \revise{We will use the NSDP for the convergence analysis of our proposed method.} 
 \begin{proposition}\cite[Theorem 3]{Titan2020}
\label{prop:NSDP-titan}
Let $\Phi:\mbbR^n\to\mbbR$ be a lower semi-continuous function and $f_i:\mbbR^n\to \mbbR\cup \{+\infty\}$ be proper lower semi-continuous functions. Denote $x^{k,0} = x^k$, 
$x^{k,i}=(x^{k+1}_1, \ldots,x^{k+1}_{i},x^{k}_{i+1},\ldots,x^{k}_s)$ for $i \in [s]$, $ 
x^{k+1} = x^{k,s}
$, and $\Psi(x):=\Phi(x) + \sum_{i=1}^s f_i(x_i)$.  Suppose $\mathcal G^k_i: \mathbb R^{ n_i} \times \mathbb R^{ n_i} \to \mathbb R^{ n_i}$ be some extrapolation operator that satisfies \revise{$\|\mathcal G^k_i(x^{k}_i, x^{k-1}_i)\|\leq a_i^k\|x^{k}_i - x^{k-1}_i\|$} for some $a_i^k$. Let  $u_i(x_i,z)$ be a block $x_i$ surrogate function of $\Phi(x)$ and 
\begin{equation}
  \label{inertial_general} 
  x_i^{k+1}=\argmin_{x_i} u_i(x_i,x^{k,i-1}) + f_i(x_i)- \langle  \mathcal G^k_i(x^{k}_i, x^{k-1}_i),x_i\rangle.
  \end{equation}
    \revise{If} one of the following conditions holds:
 \begin{itemize}
 \item[(i)] $x_i\mapsto u_i(x_i,z) + f_i(x_i)$ is $\rho_i$-strongly convex,
 \item[(ii)] the approximation error $ e_i(x_i,z):=u_i(x_i,z)-\Phi(x_i,z_{\ne i}) $  satisfying $$e_i(x_i,z)\geq \frac{\rho_i(z)}{2} \|x_i-z_i\|^2   \mbox{for all} \,\, x_i,$$ 
\end{itemize}  
(note that $\rho_i(z)$ may depend on $z$), \revise{then} we have the following NSDP
  \begin{equation}
  \label{NSDP-titan}
  \Psi(x^{k,i-1}) + \gamma_i^k \|x_i^k-x_i^{k-1} \|^2 \geq \Psi(x^{k,i}) + \eta_i^k \|x_i^{k+1}-x_i^{k} \|^2,
 \end{equation}
where 
$$
\begin{array}{ll}
\gamma^{k}_i=\frac{(a^k_i)^2}{2\nu \rho_i^{k} } , \qquad\eta^{k}_i = \frac{(1-\nu)\rho_i^{k}}{2},
\end{array}
$$ 
 $\rho_i^{k}=\rho_i(x^{k,i-1})$, and $\nu\in (0,1)$ is a constant. If we do not apply extrapolation, that is $a_i^k=0$, then \eqref{NSDP-titan} is satisfied with $\gamma_i^k=0$ and $\eta_i^k = \rho^k_i/2$. 
\end{proposition}
The following proposition, which is derived from \cite[Remark 3]{Hien_ICML2020} and \cite[Lemma 2.1]{Xu2013}, provides an NSDP when $\Phi$ is multi-convex, Lipschitz gradient suggorates are used for $\Phi$,  and $f_i$, $i \in [s]$, are also convex. We then obtain better values of $\gamma_i^k$ for the NSDP and larger extrapolation parameters could be used.  
\begin{proposition}
\label{prop:Lsmooth-strong}
Let $\Phi$, $f_i$ and $\Psi$ be defined as in Proposition \ref{prop:NSDP-titan}. For $i\in [s]$, suppose $x_i\mapsto \Phi(x)$ is an $L_i(x_{\ne i})$-smooth convex function (the Lipschitz constant may depend on $x_{\ne i}$, the values of the other blocks) and $f_i(x_i)$ is convex. Define $\hat x_i^k=x_i^k + \alpha_i^k (x_i^k-x_i^{k-1})$, $\bar x_i^k=x_i^k + \zeta_i^k (x_i^k-x_i^{k-1})$, and  $\bar x^{k,i-1}=(x^{k+1}_1,\ldots,x^{k+1}_{i-1},\bar x^{k}_i, x^{k}_{i+1},\ldots,x^k_s)$, where $\alpha_i^k$ and $\zeta_i^k$ are extrapolation parameters. Suppose 
\begin{equation} 
\label{eq:nesterov}
x_i^{k+1}=\argmin_{x_i} \langle \nabla_i \Phi(\bar x^{k,i-1}),x_i\rangle + f_i(x_i)+ \frac{L_i^k}{2}\|x_i -\hat x_i^k\|^2, 
\end{equation}
where $L_i^k=L_i(x^{k,i-1}_{\ne i}) $. Note that \eqref{eq:nesterov} is equivalent to \eqref{inertial_general} with $$u_i(x_i,x^{k,i-1})=\Phi(x^{k,i-1}) +  \langle \nabla_i \Phi(x^{k,i-1}),x_i-x_i^{k}\rangle + \frac{L_i^k}{2}\|x_i -x_i^k\|^2,$$
and $ \mathcal G_i^k(x_i^k,x_i^{k-1})=  \nabla_i \Phi(\bar x^{k,i-1})- \nabla_i \Phi( x^{k,i-1}) + L_i^k(x_i^k-\hat x_i^k)$. 
Then we have 
$$\Phi(x^{k,i-1}) + f_i(x^k_i)+\gamma_i^k \|x_i^k - x_i^{k-1}\|^2 \geq \Phi(x^{k,i}) + f_i(x^{k+1}_i) + \eta_i^k\|x^{k+1}_i-x_i^k\|^2,
$$
where $$ \gamma^{k}_i=\frac {L^k_i}{2} \big((\zeta_i^k)^2 + \frac{(\gamma_i^k-\alpha_i^k)^2}{\nu} \big), \qquad\eta^{k}_i = \frac{(1-\nu)L^k_i}{2 }.
$$
It implies that the NSDP \eqref{NSDP-titan} is also satisfied. 

If $\alpha_i^k=\zeta_i^k$ then we have Inequality \eqref{NSDP-titan} is satisfied with 
$$\gamma^{k}_i=\frac {L^k_i}{2} (\zeta_i^k)^2  , \qquad\eta^{k}_i = \frac{L^k_i}{2 }. $$   
\end{proposition} 

\subsection{Augmented Lagrangian and stationary point}
Considering Problem \eqref{CP-main}, its augmented Lagrangian is
\begin{equation}
\label{augmentedL}
\mL_{\beta}(x,y,\omega)=\Theta(x,y)+\langle h(x)+\mB y, \omega\rangle+\frac{\beta}{2}\|h(x)+ \mB y\|^{2}.
\end{equation}
\begin{definition}
 We call $(x^*,y^*)$ a stationary point of Problem \eqref{CP-main} if there exists $\omega^*$ such that  the following conditions are satisfied.
 \begin{equation}
        \label{cond:first-order}
    0 \in \partial_{x_i} (f_i(x_i^*)+ F(x^*)) + \nabla_{x_i} h(x^*)^\top \omega^*, \,\,
     \nabla G(y^*) +   \mB^{\revise{\top}} \omega^*=0,   
     \,\,h( x^*) +\mB y^*=0. 
    \end{equation}   
\end{definition}
We know that finding a stationary point of \eqref{CP-main} is equivalent to finding a critical point of the augmented Lagrangian $\mL_\beta$, see \cite[Section 2.2]{mADMM2022}. We are also interested in an $\varepsilon$-stationary point of \eqref{CP-main}, which is defined as follows.
\begin{definition}
     We call $(x^*,y^*)$ an $\varepsilon$-stationary point of Problem \eqref{CP-main} if there exists $\omega^*$ and $\chi_i\in \partial_{x_i} (f_i(x_i^*)+F(x^*))$ such that  
\vspace{-0.05in}
$$ \| \chi_i + \nabla_{x_i} h(x^*)^\top \omega^*\| \leq \varepsilon, \| \nabla G(y^*) +   \mB^{\revise{\top}} \omega^*\|\leq \varepsilon,   \| h( x^*) +\mB y^*\|\leq \varepsilon.
\vspace{-0.05in}
$$
\end{definition}

\section{Algorithm description and convergence analysis}
\label{sec:analysis}
Before describing iADMMn and presenting its convergence analysis, let us have a discussion on the proof methodology for a convergence guarantee of an alternating direction method of multipliers for solving a \emph{non-convex} optimization problem. An informal general proof recipe that describes the main steps to prove the global convergence of a first-order method to a critical point of the objective function of a \emph{non-convex} optimization problem \revise{without coupling constraints} was introduced in \cite{Attouch2013,Bolte2014}. The three main steps of the recipe are (i) sufficient decrease property of the objective sequence, (ii) a subgradient lower bound for the iterates gap, and (iii) using the K{\L} property, see \cite[Section 3.2]{Bolte2014} for more details. This proof methodology has been used wisely to prove the global convergence of the proximal point algorithm, the forward-backward splitting algorithm, the regularized Gauss–Seidel method  \cite{Attouch2009,Attouch2013}, the proximal alternating linearized minimization algorithm \cite{Bolte2014}, and the multi-block Bregman proximal alternating linearized minimization algorithm \cite{ahookhosh2019multi}. By replacing the sufficient decrease property of the objective sequence in Step (i) with the sufficient decrease property of the augmented Lagrangian sequence or the regularized augmented Lagrangian sequence, the proof methodology of \cite{Attouch2013,Bolte2014} is developed to prove the  
global convergence of  ADMM for solving non-convex problems with linear/nonlinear coupling constraints, see \cite[Section 3]{Wang2019}, \cite[Section 3.3]{Bolte2018_ADMM}, \cite[Theorem 1]{Bot2020}, \cite[Theorem 2]{mADMM2022}, \cite[Section 3]{Yashitini2022}. The augmented Lagrangian function or its regularized/auxiliary form plays the role of a Lyapunov function (sometimes it is called a potential function) of an original Lyapunov methodology. The sufficient decrease property of the Lyapunov function, the subgradient lower bound for the iterate gaps, together with the K{\L} property of the Lyapunov function are sufficient to prove the global convergence of the generated sequence to a critical point of the Lyapunov function. This general proof recipe is also applied when an inertial technique is used to accelerate the convergence, see \cite{Hien_iADMM,Hien_ICML2020,Titan2020,Hien_BMME2022,Ochs2014,Xu2013}. Choosing the appropriate Lyapunov function and establishing its sufficient decrease property is the cornerstone of proving global convergence.  In this paper, we will also use this general proof recipe to prove the global convergence of iADMMn. To that end, in the following, we sequentially describe the update of $x_i$, $y$, and $\omega$, and establish the necessary NSDPs. These NSDPs will be used to prove the sufficient decrease property of an appropriate Lyapunov function that is defined later in \eqref{lyapunov}.

\subsection{Algorithm description and the NSDPs}
\label{sec:alg_describe}
Let us remind the notations
$$x^{k,0} = x^k,\,  \, 
x^{k,i}=(x^{k+1}_1, \ldots,x^{k+1}_{i},x^{k}_{i+1},\ldots,x^{k}_s) \, \text{ for } i \in [s], 
\,\text{ and }\, \, x^{k+1} = x^{k,s}.
$$
\revise{Note that the augmented Lagrangian in \eqref{augmentedL} of Problem~\eqref{CP-main} can be rewritten as} 
$$    \mL_{\beta}(x,y,\omega)=  \sum_{i=1}^s f_i(x_i) + \varphi(x,y,\omega),
$$
where 
\[\begin{split}
    \varphi(x,y,\omega) &= F(x) +G(y) + \langle h(x)+\mB y, \omega\rangle+ \frac{\beta}{2} \|h(x)+\mB y\|^2\\
    &= F(x) +G(y) + \langle h(x)+\mB y, w\rangle+\frac{\beta}{2}\|h(x)\|^2 +\frac{\beta}{2}\|\mB y\|^{2} + \beta \langle h(x), \mB y \rangle.
\end{split}
    \]
    We summarize the description of iADMMn in Algorithm \ref{alg:iADMMn}.   The detailed update rules of $x_i$ and $y$  are elaborated in the following. 
    
\begin{algorithm}[ht!]
\caption{iADMMn for solving Problem~\eqref{CP-main}} 
\label{alg:iADMMn} 
\begin{algorithmic} 
\STATE Suppose $\hat u_i$ is a block surrogate function of $F$ with respect to block $x_i$.
\STATE Parameters of the algorithm are chosen as in \eqref{parameter}.
\STATE Choose initial points $x^{0}$, $y^0$, $\omega^0$. 
\FOR{$k=0,\ldots$}
\FOR{$i=1,\ldots,s$}
   \STATE\label{step-update} Compute the extrapolation point  $\bar x_i^k=x_i^{k} + \alpha_i^k (x_i^k-x_i^{k-1})$. 
\STATE Under Assumption \ref{assump:h2}(i), we update  $x_i$ as in \eqref{eq:xi_update_nosmooth}.
\STATE If together with Assumption \ref{assump:h2}(i) we also have $x_i\mapsto F(x_i,y_{\ne i})$ is $L_i(y_{\ne i})$-smooth then we update $x_i$ as in   \eqref{eq:xi_update}, and if Assumption \ref{assump:h2}(i) is not satisfied but   $x_i\mapsto F(x_i,y_{\ne i})$ is $L_i(y_{\ne i})$-smooth then we update $x_i$ as in  \eqref{eq:xi_update_2}. 
      \ENDFOR
   \STATE Update $y$ as in \eqref{eq:y_update}. 
    \STATE Update $\omega$ as  \begin{equation}
   \label{eq:omega_update}
   \omega^{k+1}=\tau_1 \omega^k  + \tau_2\beta (h( x^{k+1})+ \mathcal B y^{k+1}). 
   \end{equation} 
\ENDFOR
\end{algorithmic}
\end{algorithm}
\paragraph{Update rule for block $x_i$ and NSDP when updating $x_i$} Following the inertial technique proposed in \cite{Titan2020}, we conduct three steps to update block $x_i$: (1) form a block $x_i$ surrogate function of   $\varphi(x,y,\omega)$, (2) add an inertial term $
\langle\mG_i^k(x_i^k,x_i^{k-1}),x_i \rangle$ to this surrogate function, and (3) update $x_i$ by minimizing the sum of $f_i(x_i)$ and the inertial surrogate function. 
To form a block $x_i$ surrogate function of   $\varphi(x,y,\omega)$, we note that if $u_i(x_i,z)$ is a block $x_i$ surrogate function of $F(x) + \frac{\beta}{2} \|h(x)\|^2$  then  
\begin{equation}
    \label{surrogate_of_alpha}
    u_i(x_i,z)+G(y)+ \langle h(x_i,z_{\ne i})+\mB y, w\rangle +\frac{\beta}{2}\|\mB y\|^{2}+ \beta \langle h(x_i,z_{\ne i}), \mB y \rangle 
\end{equation} is a block $x_i$ surrogate function of    $\varphi(x,y,\omega)$. A general update rule for $x_i$ can be described as follows.
\begin{equation}
     \label{eq:update_xi_general}
     \begin{split}
 & x_i^{k+1} \in\argmin_{x_i} \Big\{ f_i(x_i)  +  \langle h(x_i,x^{k,i-1}_{\ne i}),  \omega^k + \beta\mB y^k\rangle + u_i(x_i,x^{k,i-1}) \\
 &\qquad\qquad\qquad \qquad\qquad-\langle \mG_i^k(x_i^k,x_i^{k-1}),x_i\rangle \Big\}.
\end{split}
 \end{equation}
Denote $h_2(x)=\frac12\|h(x)\|^2$. As we assume  $x_i\mapsto h_2(x_i,y_{\ne i})$ is $l_i(y_{\ne i})$-smooth  (see Assumption \ref{assump:h2}(i)), then we take the Lipschitz gradient surrogate for $\frac12\|h(x)\|^2$ and formulate $u_i$ of \eqref{surrogate_of_alpha}  as follows
 \begin{equation} 
 \label{ui1}
 u_i(x_i,z)=\hat u_i(x_i,z)+ \beta h_2(z)+ \beta\langle \nabla_i h_2(z),x_i-z_i \rangle + \frac{\beta\kappa_i(z_{\ne i})}{2}\|x_i-z_i\|^2, 
 \end{equation}
  where $\kappa_i(z_{\ne i})\geq  l_i(z_{\ne i})$ and $\hat u_i$ is a block surrogate function of $F$ with respect to $x_i$. 
After forming the block surrogate for $\varphi$,  the second step is to take the inertial term 
  \begin{equation}
  \label{mG1}
  \mG_i^k(x_i^k,x_i^{k-1})=\beta\nabla_i h_2 (x_i^k,x^{k,i-1}_{\ne i})-\beta\nabla_i h_2(\bar x_i^k,x^{k,i-1}_{\ne i}) + \beta \kappa_i^k (\bar x_i^k-x_i^k),
  \end{equation}
     where $\kappa_i^k \geq l_i(x^{k,i-1}_{\ne i}) = l_i^k$ and  $\bar x_i^k=x_i^{k} + \alpha_i^k (x_i^k-x_i^{k-1})$, here $\alpha_i^k$ is an extrapolation parameter.
         Then the update \eqref{eq:update_xi_general} of $x_i$ is rewritten as follows.
\begin{equation}
    \label{eq:xi_update_nosmooth}
    \begin{split}
 & x_i^{k+1} \in\argmin_{x_i} \Big\{ f_i(x_i) + \hat u_i(x_i,x^{k,i-1}_{\ne i})  +  \langle h(x_i,x^{k,i-1}_{\ne i}),  \omega^k + \beta\mB y^k\rangle
 \\  
  &\,+ \big\langle \beta \nabla_i h(\bar x_i^k,x^{k,i-1}_{\ne i})^\top h(\bar x_i^k,x^{k,i-1}_{\ne i}),x_i \big\rangle
        +\frac{\beta\kappa_i^k}{2}\|x_i- \bar x_i^k\|^2\Big\}.
\end{split}
\end{equation}
Let us establish the NSDP for $\mL_\beta$ when the update in \eqref{eq:xi_update_nosmooth} is used. We have 


$$\|\mG_i^k(x_i^k,x_i^{k-1})\|\leq \beta (l_i^k + \kappa_i^k) \alpha_i^k \|\Delta x_i^k\|. 
$$
The approximation error between $\varphi(x,y,\omega)$ and its block $x_i$ surrogate in \eqref{surrogate_of_alpha} satisfies that 
\[ \begin{split}
    &\hat u_i(x_i,z)-F(x_i,z_{\ne i}) +  \beta h_2(z)+ \beta\langle \nabla_i h_2(z),x_i-z_i \rangle \\
    & \qquad \qquad\qquad\qquad\qquad+ \frac{\kappa_i(z_{\ne i})\beta}{2}\|x_i-z_i\|^2 -  \beta h_2(x_i,z_{\ne i}) \\
    &\geq \theta_i(x_i,z):=\beta h_2(z)  -  \beta h_2(x_i,z_{\ne i})  + \beta\langle \nabla_i h_2(z),x_i-z_i \rangle + \frac{\kappa_i(z_{\ne i})\beta}{2}\|x_i-z_i\|^2.
\end{split}
\]
Note that $\nabla^2_{x_i} \theta_i(x_i,z)=\kappa_i (z_{\ne i})\beta \mathcal I - \beta \nabla^2_{x_i} h_2(x_i,z_{\ne i})$. This implies $x_i\mapsto \theta_i(x_i,z)$
is $\beta(\kappa_i(z_{\ne i})-l_i(z_{\ne i}))$-strongly convex (since $\nabla^2_{x_i} h_2(x_i,z_{\ne i}) \preceq  l_i(z_{\ne i}) \mathcal I$). On the other hand, $\nabla_{x_i} \theta_i(z_i,z)=0$. It follows from \cite[Lemma 1]{Hien_iADMM} that $$\theta_i(x_i,z)\geq \frac{1}{2}\beta(\kappa_i(z_{\ne i})-l_i(z_{\ne i})) \|x_i-z_i\|^2.$$
We then apply Proposition \ref{prop:NSDP-titan} to obtain the following NSDP. 
 \begin{equation}
     \label{NSDP_1}
\revise{\mL_\beta}(x^{k,i},y^k,\omega^k)+\eta^k_i \|\Delta x^{k+1}_i\|^2
\leq
\revise{\mL_\beta}(x^{k,i-1},y^k,\omega^k)+ \gamma_i^k \|\Delta x^k_i\|^2,
 \end{equation}
 where $\kappa_i^k > l_i^k$ and 
\begin{equation}
\label{eta_gamma1}
\eta^k_i = \frac{(1-\nu_i)(\kappa^k_i- l_i^k)\beta}{2}, \quad
\gamma_i^k= \frac{(a_i^k)^2}{2\nu_i(\kappa^k_i- l_i^k)\beta}, \quad a_i^k=\beta(l_i^k+\kappa_i^k)\alpha_i^k. 
\end{equation}
If $x_i\mapsto f_i(x_i)+ \hat u_i(x_i,x^{k,i-1}_{\ne i}) $ 
is convex and $x_i \mapsto h(x)$ is linear then we take $\kappa_i^k=l_i^k$ and the NSDP in \eqref{NSDP_1} can be tighter. Indeed, as $x_i\mapsto h_2(x)$ and $x_i\mapsto f_i(x_i)+ \hat u_i(x_i,x^{k,i-1}_{\ne i})+ \langle h(x_i,x^{k,i-1}_{\ne i}),  \omega^k + \beta\mB y^k\rangle $ are convex, we  apply Proposition \ref{prop:Lsmooth-strong} to obtain
\begin{equation}
\label{temp1}
\begin{split}
   & h_2(x^{k,i})+f_i(x_i^k)+ \hat u_i(x_i^k,x^{k,i-1}_{\ne i})  +  \langle h(x_i^k,x^{k,i-1}_{\ne i}),  \omega^k + \beta\mB y^k\rangle+\eta_i^k \|\Delta x_i^{k+1}\|^2 \\
   & \leq h_2(x^{k,i-1}) +f_i(x_i^{k-1})+ \hat u_i(x_i^{k-1},x^{k,i-1}_{\ne i})  +  \langle h(x_i^{k-1},x^{k,i-1}_{\ne i}),  \omega^k + \beta\mB y^k\rangle\\
   &\qquad\qquad+\gamma_i^k \|\Delta x_i^{k}\|^2,
   \end{split}
\end{equation}
where $\gamma_i^k=\frac12 \beta l_i^k (\alpha_i^k)^2$ and $\eta_i^k=\frac12 \beta l_i^k$. 
Furthermore, note that $\hat u_i(x_i^{k-1},x^{k,i-1}_{\ne i}) =F(x^{k,i-1})$ and $ \hat u_i(x_i^k,x^{k,i-1}_{\ne i}) \geq  F(x_i^k,x^{k,i-1}_{\ne i})= F(x^{k,i})$. Therefore, \eqref{temp1} implies that \eqref{NSDP_1} is satisfied with  $\gamma_i^k=\frac12 \beta l_i^k (\alpha_i^k)^2$ and $\eta_i^k=\frac12 \beta l_i^k$. 
In the following proposition, we summarize the NSDP when using \eqref{eq:xi_update_nosmooth} to update $x_i$ . 

\begin{proposition}[NSDP when updating $x_i$] \label{prop:nsdp_xi}
Suppose Assumption \ref{assump:h2}(i) is satisfied and $x_i$ is updated as in \eqref{eq:xi_update_nosmooth}. We choose $\kappa_i^k>l_i^k$, then the NSDP \eqref{NSDP_1} is satisfied with $\eta_i^k$ and $\gamma_i^k$ defined in \eqref{eta_gamma1}. If $x_i\mapsto f_i(x_i)+ \hat u_i(x_i,x^{k,i-1}_{\ne i})  $ is convex and $x_i \mapsto h(x)$ is linear then we take $\kappa_i^k = l_i^k$ and then the NSDP \eqref{NSDP_1} is satisfied with   $\gamma_i^k=\frac12 \beta l_i^k (\alpha_i^k)^2$ and $\eta_i^k=\frac12 \beta l_i^k$.
\end{proposition}
We discuss two other situations in the following remarks. 
\begin{remark}
\label{remark-2}
Consider the case that together with Assumption \ref{assump:h2}(i)  we also have $x_i\mapsto F(x_i,y_{\ne i})$ is an $L_i(y_{\ne i})$-smooth function. Then $x_i \mapsto \hat h(x):=F(x) + \frac{\beta}{2}\|h(x)\|^2$ is $\mathbf l_i(y_{\ne i})=L_i(y_{\ne i})+\beta l_i(y_{\ne i})$-smooth. We take the Lipschitz gradient surrogate for $F(x)+\frac{\beta}{2}\|h(x)\|^2$, that is, $u_i$ in \eqref{surrogate_of_alpha} is
        \begin{equation} 
        \label{ui2}
         u_i(x_i,z)=\hat h(z)+\langle \nabla_i\hat h(z),x_i-z_i \rangle + \frac{\kappa_i(z_{\ne i})}{2}\|x_i-z_i\|^2, 
         \end{equation}
        where $\kappa_i(z_{\ne i})\geq \mbfl_i(z_{\ne i})$.  And we take the inertial term
      \begin{equation}
      \label{mG2}
      \mG_i^k(x_i^k,x_i^{k-1})=\nabla_i \hat h(x_i^k,x^{k,i-1}_{\ne i})-\nabla_i \hat h(\bar x_i^k,x^{k,i-1}_{\ne i}) + \kappa_i^k (\bar x_i^k-x_i^k),
      \end{equation}
        where $\kappa_i^k\geq \mbfl_i^k =\mbfl_i(x^{k,i-1}_{\ne i})$. The update \eqref{eq:update_xi_general} of $x_i$ is rewritten as follows 
        \begin{equation}
\label{eq:xi_update} 
\begin{split}
 & x_i^{k+1} \in\argmin_{x_i} \Big\{ f_i(x_i)  +  \langle h(x_i,x^{k,i-1}_{\ne i}),  \omega^k + \beta\mB y^k\rangle
 \\  
  &+ \big\langle \nabla_i F(\bar x_i^k,x^{k,i-1}_{\ne i})+\beta \nabla_i h(\bar x_i^k,x^{k,i-1}_{\ne i})^\top h(\bar x_i^k,x^{k,i-1}_{\ne i}),x_i \big\rangle
        +\frac{\kappa_i^k}{2}\|x_i- \bar x_i^k\|^2\Big\}.
\end{split}
 \end{equation}
 Similarly to the above reasoning of Proposition \ref{prop:nsdp_xi}, we take $\kappa_i^k > \mathbf l_i^k $  and apply Proposition \ref{prop:NSDP-titan} to obtain the NSDP in \eqref{NSDP_1} with 
 \begin{equation}
 \label{eta_gamma2}
\eta^k_i = \frac{(1-\nu_i)(\kappa_i^k- \mathbf l_i^k)\beta}{2}, \quad
\gamma_i^k= \frac{(a_i^k)^2}{2\nu_i(\kappa_i^k- \mathbf  l_i^k)\beta}, \quad a_i^k=\beta(\mathbf  l_i^k+\kappa_i^k)\alpha_i^k. 
\end{equation}
If together with the smoothness of $x_i\mapsto \hat h(x)$ we assume that $x_i \mapsto  f_i(x_i)$ and $x_i\mapsto F(x)$ are convex and $x_i\mapsto h(x)$ is linear, then we take $\kappa_i^k=\mathbf l_i^k$ and \eqref{NSDP_1} holds with $\gamma_i^k=\frac12 \beta \mathbf l_i^k (\alpha_i^k)^2$ and $\eta_i^k=\frac12 \beta  \mathbf l_i^k$. 
 \end{remark}

\begin{remark}
\label{remark_1} 
In this remark, we discuss a case when Assumption \ref{assump:h2}(i) can be removed: we assume $x_i\mapsto F(x_i,y_{\ne i})$ is an $L_i(y_{\ne i})$-smooth function (but $x_i\mapsto h_2(x_i,y_{\ne i})$ can be non-smooth),  then we take
\begin{equation} 
\label{ui3}
u_i(x_i,z)= \beta h_2(x_i,z_{\ne i})+ F(z)+ \langle \nabla_i F(z),x_i-z_i \rangle + \frac{\kappa_i (z_{\ne i})}{2}\|x_i-z_i\|^2,
\end{equation}
where $\kappa_i(z_{\ne i})\geq L_i(z_{\ne i})$. And  we take the inertial term
\begin{equation}
      \label{mG3}
      \mG_i^k(x_i^k,x_i^{k-1})=\nabla_i F (x_i^k,x^{k,i-1}_{\ne i})-\nabla_i F(\bar x_i^k,x^{k,i-1}_{\ne i}) + \kappa_i^k (\bar x_i^k-x_i^k),
\end{equation}
where $\kappa_i^k \geq L_i(x^{k,i-1}_{\ne i})=L_i^k$.  The update \eqref{eq:update_xi_general} of $x_i$ is rewritten as follows  \begin{equation}
\label{eq:xi_update_2} 
\begin{split}
 & x_i^{k+1} \in\argmin_{x_i} \Big\{ f_i(x_i)  +  \langle h(x_i,x^{k,i-1}_{\ne i}),  \omega^k + \revise{\beta}\mB y^k\rangle + \frac{\beta}{2}\|h(x_i,x^{k,i-1}_{\ne i}) \|^2
 \\  
  &\,+ \big\langle \nabla_i F(\bar x_i^k,x^{k,i-1}_{\ne i}),x_i \big\rangle
        +\frac{\kappa_i^k}{2}\|x_i- \bar x_i^k\|^2\Big\}.
\end{split}
 \end{equation}
 Similarly to the reasoning of Proposition \ref{prop:nsdp_xi}, we take $\kappa_i^k>L_i^k$ and apply Proposition \ref{prop:NSDP-titan} to obtain the NSDP in \eqref{NSDP_1} with 
 \begin{equation}
 \label{eta_gamma3}
\eta^k_i = \frac{(1-\nu_i)(\kappa_i^k- L_i^k)\beta}{2}, \quad
\gamma_i^k= \frac{(a_i^k)^2}{2\nu_i(\kappa_i^k- L_i^k)\beta}, \quad a_i^k=\beta(  L_i^k+\kappa_i^k)\alpha_i^k. 
\end{equation}
If together with the smoothness of $x_i\mapsto F(x)$ we assume that $x_i \mapsto  f_i(x_i)$ and $x_i\mapsto F(x)$ are convex and $x_i\mapsto h(x)$ is linear, then we take $\kappa_i^k=L_i^k$ and \eqref{NSDP_1} holds with $\gamma_i^k=\frac12 \beta L_i^k (\alpha_i^k)^2$ and $\eta_i^k=\frac12 \beta  L_i^k$. 
    \end{remark}

\paragraph{Update $y$} As $G$ is $L_G$-smooth, we form a  block $y$ surrogate of $\varphi$ by summing the Lipschitz gradient surrogate of $G$ and the remaining part of $\varphi$:
\[ \begin{split} 
\hat\varphi(y,x,y',\omega)&=G(y')+\langle \nabla G(y'),y-y'\rangle + \frac{L_G}{2}\|y-y'\|^2 \\
&\qquad+ F(x)+ \langle h(x)+\mB y, \omega\rangle+ \frac{\beta}{2}\|h(x)+\mB y\|^2. 
\end{split}
\]
Block $y$ is updated as follows
\begin{equation}
   \label{eq:y_update}
   \begin{split}
y^{k+1}&\in\argmin_{y}\hat\varphi(y,x^{k+1},y^k,\omega^k)\\
       &=\argmin_{y} \langle\revise{\mB^\top}\omega^k +\nabla G(y^k),y \rangle + \frac{\beta}{2} \|h( x^{k+1}) + \mathcal  By\|^2+\frac{L_G}{2}\|y- y^k\|^2,\\
       &=(\beta \mB^{\revise{\top}}\mB + L_G \mathbf I)^{-1}\big(L_G y^k - \nabla G(y^k)-\mB^{\revise{\top}}(\omega^k + \beta h(x^{k+1}))\big).
        \end{split}
   \end{equation}
   \begin{proposition}[Sufficient decrease when updating $y$]
   \label{ynsdp}
  The update in \eqref{eq:y_update} guarantees a sufficient decrease
\begin{equation}
\label{eq:r2}
\revise{\mL_\beta}(x^{k+1},y^{k+1},\omega^k)+ \frac{\delta}{2}\|y^{k+1}-y^k\|^2 \leq \revise{\mL_\beta}(x^{k+1},y^k,\omega^k),
\end{equation}     
where $\delta=L_G+\beta \lambda_{\min}(\mB^{\revise{\top}}\mB)$. If $G$ is convex, then \eqref{eq:r2} is satisfied with $\delta=L_G$.
   \end{proposition}
   \begin{proof}
       We note that $y\mapsto \hat \varphi(y,x,y',\omega)$ is $L_G+\beta \lambda_{\min}(\mB^{\revise{\top}}\mB)$-strongly convex. Applying Proposition \ref{prop:NSDP-titan} (note that we do not use inertial term for updating $y$), we obtain \eqref{eq:r2}.  If $G$ is convex, then we apply Proposition \ref{prop:Lsmooth-strong} (we simply take $\alpha_i^k=\beta_i^k=0$ in Proposition  \ref{prop:Lsmooth-strong}) to obtain the result. 
   \end{proof}

 \subsection{Convergence analysis} 
 
Given $\tau_1$, $\tau_2$, we denote $C_1=\frac{(\tau_1 +1)|\tau_1 -\tau_2|}{2\sigma_{\mB}\tau_2\beta (1-|\tau_1 -\tau_2|)}$,
$C_2=\frac{(\tau_1 +1)\tau_2/\tau_1 }{2\sigma_{\mB}\beta(1-|\tau_1 -\tau_2|)(1-|1-\tau_2/\tau_1 |)}$, $C_3=\frac{\delta}{2}-2C_2 L_G^2$, where $\delta$ is defined in Proposition \ref{ynsdp}. \revise{Note that $\gamma_i^k$ and $\eta_i^k$ are the coefficients in the NSDP\eqref{NSDP_1}, their formulas are summarized in Table~\ref{tab:my_label}.

\begin{table}[ht]
     \caption{ \revise{$\gamma_i^k$ and $\eta_i^k$ in the NSDP\eqref{NSDP_1}}}
    \label{tab:my_label}\center
\begin{tabular}{ |l| }
\hline\\
\revise{$\gamma_i^k$ and $\eta_i^k$  are defined in: }\\
\quad \revise{- Proposition \ref{prop:nsdp_xi} if we assume Assumption \ref{assump:h2}(i)  and use the update rule in \eqref{eq:xi_update_nosmooth},}
    \\
\quad \revise{- Remark \ref{remark-2} if we assume  Assumption \ref{assump:h2}(i) together with $x_i\mapsto F(x_i,y_{\ne i})$  being}\\ 
\qquad\qquad \revise{an $L_i(y_{\ne i})$-smooth function and use the update rule in \eqref{eq:xi_update},}
   \\
  \quad \revise{- Remark \ref{remark_1} if Assumption~\ref{assump:h2}(i) is not satisfied but $F(x_i,y_{\ne i})$ is an $L_i(y_{\ne i})$-smooth}\\  \qquad\qquad   \revise{function and we use the update rule in \eqref{eq:xi_update_2}}. ~\\
\hline    
\end{tabular}
\end{table}}
We need to choose the parameters such that they satisfy the following conditions.
 \begin{equation}
 \label{parameter}
 \begin{split}
  &\tau_1\in (0,1],\quad   \tau_2/\tau_1 \in (0,2),\quad |\tau_1-\tau_2|<1\\
  &\gamma_i^k \leq B_1 \revise{\eta_i^{k-1}}, \quad 8 C_2 L_G^2 \leq B_2 C_3,
 \end{split}
 \end{equation}
where $B_1, B_2\in (0,1)$ are some constants.
In this section, we will establish the subsequential convergence and the global convergence for iADMMn.
The following proposition provides a recursive inequality for $\{\revise{\mL_\beta}(x^k,y^k;\omega^k)\}$. 
\begin{proposition}
    \label{prop:Lk-recursive} \revise{Considering Algorithm~\ref{alg:iADMMn},}
     the following inequality holds
    \[
\begin{split}
&\revise{\mL_\beta}(x^{k+1},y^{k+1},\omega^{k+1}) -\frac{1-\tau_1 }{2\tau_2 \beta} \| \omega^{k+1}\|^2+ \sum_{i=1}^s\eta_i^k \|\Delta x_i^{k+1}\|^2 + C_3 \|\Delta y^{k+1}\|^2\\
&\leq \revise{\mL_\beta}(x^{k},y^{k},\omega^{k})-\frac{1-\tau_1 }{2\tau_2 \beta} \| \omega^{k}\|^2+ \sum_{i=1}^s\gamma_i^k \|\Delta x_i^{k}\|^2  \\
&\qquad+ C_1\big(\| \revise{\mB^\top}\Delta \omega^{k}\|^2 -\|\revise{\mB^\top} \Delta \omega^{k+1} \|^2\big)+ 8C_2 L_G^2 \|\Delta y^{k}\|^2 .
\end{split}
\]
\end{proposition}
\begin{proof} 
Optimality condition of \eqref{eq:y_update} gives us
\begin{equation}
\label{eq:nabM}
\revise{\mB^\top} (\omega^k + \beta(h(x^{k+1})+\mB y^{k+1}))+\nabla G(y^k) + L_G(y^{k+1}-y^k)=0
\end{equation}
Denote $t^{k+1}=  L_G \Delta y^{k+1} +\nabla G(y^{k})$.  From \eqref{eq:nabM} and the update of $\omega$ in \eqref{eq:omega_update} we have
\begin{align*}
 \revise{\mB^\top} ((1-\frac{\tau_1}{\tau_2})\omega^k + \frac{1}{\tau_2}\omega^{k+1} )=-t^{k+1},
\end{align*}
which implies that 
\begin{equation*}
 \frac{1}{\tau_2}\revise{\mB^\top} \Delta \omega^{k+1} + (1-\frac{\tau_1}{\tau_2}) \revise{\mB^\top} \Delta \omega^{k} = -\Delta t^{k+1}.
\end{equation*}
Hence, we have 
\begin{equation}
\label{eq:gammaB}
 \frac{1}{\tau_1}\revise{\mB^\top} \Delta \omega^{k+1} =  (1-\frac{\tau_2}{\tau_1})\revise{\mB^\top} \Delta \omega^{k}  -\frac{\tau_2}{\tau_1}\Delta t^{k+1}.
\end{equation}
Note that if $ \tau_2/\tau_1 \in [1,2)$ then \eqref{eq:gammaB} can be rewritten as 
$$\frac{1}{\tau_1}\revise{\mB^\top} \Delta \omega^{k+1} =  -(\tau_2/\tau_1-1)\revise{\mB^\top} \Delta \omega^{k}  -  (2-\tau_2/\tau_1)\frac{\tau_2/\tau_1}{2-\tau_2/\tau_1}\Delta t^{k+1}.
$$ 
Hence, from $\tau_2/\tau_1 \in (0,2)$ and the convexity of the norm $\|\cdot\|$, we can derive from \eqref{eq:gammaB} the following inequality 
\begin{equation}
\label{eq:gammaB_2}
\frac{1}{\tau_1 }\|\revise{\mB^\top} \Delta \omega^{k+1} \|^2\leq |1-\tau_2/\tau_1 |\| \revise{\mB^\top}\Delta \omega^{k}\|^2 + \frac{(\tau_2/\tau_1 )^2 }{1-|1-\tau_2/\tau_1 |} \|\Delta t^{k+1}\|^2.
\end{equation}
On the other hand, from the definition of $t^{k+1}$ and the Lipschitz continuity of $\nabla G$, we have $\|\Delta t^{k+1}\| \leq L_G \|\Delta y^{k+1}\| + L_G  \|\Delta y^{k}\| + L_G  \|\Delta y^{k}\|$. It implies that  
\[
\begin{split}
 \|\Delta t^{k+1}\|^2\leq  2L_G^2 \|\Delta y^{k+1}\|^2 +  8 L_G^2 \|\Delta y^{k}\|^2.
\end{split}
\]
Hence, from \eqref{eq:gammaB_2} we obtain
\begin{equation}
\label{eq:gammaB_4}
\begin{split}
\|\revise{\mB^\top}\Delta \omega^{k+1} \|^2 &\leq \frac{|\tau_1-\tau_2|}{(1-|\tau_1-\tau_2|)}\big(\| \revise{\mB^\top}\Delta \omega^{k}\|^2 -\|\revise{\mB^\top} \Delta \omega^{k+1} \|^2\big) \\
&\quad+ \frac{\tau_2^2/\tau_1 }{(1-|\tau_1-\tau_2|)(1-|1-\tau_2/\tau_1|)}\big(2L_G^2\|\Delta y^{k+1}\|^2 + 8 L_G^2 \|\Delta y^{k}\|^2 \big).
\end{split}
\end{equation}
Furthermore, 
\[\begin{split}
&\revise{\mL_\beta}(x^{k+1},y^{k+1},\omega^{k+1})-\revise{\mL_\beta}(x^{k+1},y^{k+1},\omega^{k})\\
&=\big\langle h(x^{k+1})+\mB y^{k+1}, \omega^{k+1}- \omega^{k} \big\rangle\\
&=\big\langle\frac{1}{\tau_2\beta}(\omega^{k+1}-\tau_1  \omega^{k}), \omega^{k+1}- \omega^{k} \big\rangle\\
&=\frac{1}{\tau_2\beta}\big\langle\tau_1 (\omega^{k+1}- \omega^{k})+ (1-\tau_1 ) \omega^{k+1},\omega^{k+1}- \omega^{k} \big\rangle \\
&=\frac{\tau_1 }{\tau_2\beta} \|\Delta \omega^{k+1}\|^2 + \frac{1-\tau_1 }{2\tau_2 \beta} \big( \|\Delta \omega^{k+1}\|^2+ \| \omega^{k+1}\|^2-\| \omega^{k}\|^2\big),
\end{split}
\]
which leads to 
\begin{equation}
\label{eq:gammaB_5}
\begin{split}
&\revise{\mL_\beta}(x^{k+1},y^{k+1},\omega^{k+1})-\revise{\mL_\beta}(x^{k+1},y^{k+1},\omega^{k})\\
&=\frac{\tau_1 +1}{2\tau_2\beta}\|\Delta \omega^{k+1}\|^2 + \frac{1-\tau_1 }{2\tau_2 \beta}\big(  \| \omega^{k+1}\|^2-\| \omega^{k}\|^2\big).
\end{split}
\end{equation}
Therefore, from \eqref{eq:gammaB_4}, \eqref{eq:gammaB_5}, and noting that $\sigma_{\mB} \|\Delta \omega^{k+1}\|^2 \leq \|\mB^{\revise{\top}} \Delta \omega^{k+1}\|^2$, we have 
\[\begin{split}
&\revise{\mL_\beta}(x^{k+1},y^{k+1},\omega^{k+1})-\revise{\mL_\beta}(x^{k+1},y^{k+1},\omega^{k})
\\
&\leq C_1\big(\| \revise{\mB^\top}\Delta \omega^{k}\|^2 -\|\revise{\mB^\top} \Delta \omega^{k+1} \|^2\big) + C_2 2 L_G^2 \|\Delta y^{k+1}\|^2+ C_28L_G^2 \|\Delta y^{k}\|^2 \\
&\qquad\qquad+ \frac{1-\tau_1 }{2\tau_2 \beta}\big( \| \omega^{k+1}\|^2-\| \omega^{k}\|^2\big).
\end{split}
\]
Together with \eqref{NSDP_1} and \eqref{eq:r2} we obtain the result.
\end{proof}
\begin{lemma}
\label{lemma:Lhat}
  Denote 
    \begin{equation} 
    \label{Lhat}
    \begin{split}
    \hat{\mL}^k&= \revise{\mL_\beta}(x^{k},y^{k},\omega^{k})-\frac{1-\tau_1 }{2\tau_2 \beta} \| \omega^{k}\|^2+ \sum_{i=1}^s B_1\revise{ \eta_i^{k-1}} \|\Delta x_i^{k}\|^2  \\
    &\qquad+ C_1 \| \revise{\mB^\top}\Delta \omega^{k}\|^2 +  B_2 C_3  \|\Delta y^{k}\|^2. 
    \end{split}
    \end{equation}
    We have $\hat{\mL}^k\geq \nu$ for all $k\geq 1$, where $\nu$ is the lower bound of $\Theta(x,y)$ (see Assumption \ref{assump:theta}). 
\end{lemma}
\begin{proof}
We use the technique of \cite[Lemma 2.9]{melo2017}. From Proposition \ref{prop:Lk-recursive} and the conditions $\gamma_i^k \leq B_1 \eta_i^{k-1}$ and $8C_2 L_G^2 \leq B_2 C_3 $ for some constants $B_1, B_2 \in (0,1)$, we have $\hat{\mL}^{k+1} \leq \hat{\mL}^k$, that is, $\{\hat{\mL}^k\}$ is a non-increasing sequence. Suppose there exists $k_0$ such that $\hat{\mL}^k <\nu$ for all $k\geq k_0$, then we have 
\[
\sum_{k=1}^K (\hat \mL^k - \vartheta) \leq \sum_{k=1}^{k_0} (\hat \mL^k -\vartheta) + (K-k_0) (\hat \mL^k -\vartheta).
\]
This implies that $\sum_{k=1}^\infty (\hat \mL^k - \vartheta)= -\infty$. 
However,  since 
   \[
\begin{split}
\hat{\mL}^k &\geq \revise{\mL_\beta}(x^{k},y^{k},\omega^{k})-\frac{1-\tau_1 }{2\tau_2 \beta}\| \omega^{k}\|^2\\
&\geq \nu+\langle h(x^k)+\mB y^k, \omega^k\rangle -\frac{1-\tau_1 }{2\tau_2 \beta}\| \omega^{k}\|^2\\
&=\nu+\frac{1}{\tau_2\beta}\langle \omega^{k}-\tau_1 \omega^{k-1}, \omega^k\rangle -\frac{1-\tau_1}{2\tau_2 \beta}\| \omega^{k}\|^2\\
&=\nu+\frac{1+\tau_1 }{2\tau_2 \beta}\| \omega^{k}\|^2-\frac{\tau_1 }{\tau_2 \beta}\langle \omega^{k-1}, \omega^k\rangle\\
&=\nu+\frac{1+\tau_1 }{2\tau_2 \beta}\| \omega^{k}\|^2+\frac{\tau_1 }{2\tau_2 \beta}(\|\Delta \omega^k\|^2-\| \omega^{k}\|^2-\| \omega^{k-1}\|^2)\\
&\geq\nu+\frac{1}{2\tau_2 \beta}\| \omega^{k}\|^2-\frac{\tau_1 }{2\tau_2 \beta}\| \omega^{k-1}\|^2\\
&\geq \nu+\frac{\tau_1 }{2\tau_2 \beta}(\| \omega^{k}\|^2-| \omega^{k-1}\|^2),
\end{split}
\]
 we have
$$ \sum_{k=1}^K (\hat \mL^k - \vartheta) \geq \frac{\tau_1 }{2\tau_2 \beta} \sum_{k=1}^K(\|\omega^k\|^2 - \|\omega^{k-1}\|^2)\geq  \frac{\tau_1 }{2\tau_2 \beta} (-\|\omega^{0}\|^2), 
$$
which gives a contradiction to the fact $\sum_{k=1}^\infty (\hat \mL^k - \vartheta)= -\infty$. 
\end{proof}

\begin{proposition}
    \label{prop:delta_converge}
 \revise{Consider Algorithm~\ref{alg:iADMMn}. The values of $\gamma_i^k$ and $\eta_i^k$  are given in Table~\ref{tab:my_label}.}  Suppose there exist $\tilde\eta_i>0$ such that $\eta_i^k\geq \tilde\eta_i$ for all $i\in [s]$ and $k\geq 1$. Then the sequence $\{\Delta y^k\}_{k\geq 0}$, $\{\Delta x_i^k\}_{k\geq 0}$, and $\{\Delta \omega^k\}_{k\geq 0}$ converge to 0. 
\end{proposition}
\begin{proof}
   From Proposition \ref{prop:Lk-recursive} we have 
\begin{equation}
\label{eq:Lhatrecursive}
    \hat{\mL}^{k+1} + (1-B_1)\sum_{i=1}^s\eta_i^k\|\Delta x_i^{k+1} \|^2 + (1-B_2)C_3 \|\Delta y^{k+1}\|^2 \leq \hat{\mL}^k,
\end{equation}
where $\hat{\mL}^k$ is defined in \eqref{Lhat}. For all $K\geq 1$, by summing \eqref{eq:Lhatrecursive} from $k=0$ to $K$ we get
\begin{equation}
\label{eq:Lhat_lower} \hat{\mL}^{K+1} + (1-B_1) \sum_{k=0}^K \sum_{i=1}^s\eta_i^k\|\Delta x_i^{k+1} \|^2 +  (1-B_2)C_3 \sum_{k=0}^K\|\Delta y^{k+1}\|^2 \leq\hat{\mL}^0.
\end{equation}
From \eqref{eq:Lhat_lower}, Lemma \ref{lemma:Lhat} and the assumption $\eta_i^k \geq\tilde\eta_i>0$ we derive that the sequence $\{\Delta y^k\}_{k\geq 0}$ and $\{\Delta x_i^k\}_{k\geq 0}$ converge to 0. 

From \eqref{eq:gammaB_4}, we have 
\[\begin{split}
&\sum_{k=0}^{K} \|\revise{\mB^\top}\Delta \omega^{k+1} \|^2  \leq \frac{|\tau_1-\tau_2|}{(1-|\tau_1-\tau_2|)} \| \revise{\mB^\top}\Delta \omega^{0}\|^2 \\
&\quad+ \frac{\tau_2^2/\tau_1 }{(1-|\tau_1-\tau_2|)(1-|1-\tau_2/\tau_1|)}\big(2L_G^2\sum_{k=0}^{K}\|\Delta y^{k+1}\|^2 + 8 L_G^2\sum_{k=0}^{K} \|\Delta y^{k}\|^2 \big).
\end{split}
\]
It implies that $\sum_{k=0}^{\infty}\|\Delta \omega^{k} \|^2 \leq \sum_{k=0}^{\infty} 1/\sigma_{\mB}\|\revise{\mB^\top}\Delta \omega^{k} \|^2 <+\infty$. Hence $\{\Delta \omega^k\}_{k\geq 0}$ converge to 0.
\end{proof}

\begin{theorem}[Subsequential convergence] 
\label{theorem-sub}
\revise{Consider Algorithm~\ref{alg:iADMMn} and the parameters are chosen to satisfy the conditions in \eqref{parameter}. The values of $\gamma_i^k$ and $\eta_i^k$  are summarized in Table~\ref{tab:my_label}.} Suppose there exist $\tilde\eta_i>0$ such that $\eta_i^k\geq \tilde\eta_i$ for all $i\in [s]$ and $k\geq 1$. And suppose  
 $a_i^k$ \revise{(which are defined in \eqref{eta_gamma1} if we use the update rule in \eqref{eq:xi_update_nosmooth}), defined in \eqref{eta_gamma2} if we use use the update rule in \eqref{eq:xi_update}, and defined in \eqref{eta_gamma3} if we use use the update rule in \eqref{eq:xi_update_2})}  are upper bounded\footnote{In fact, if the Lipschitz constants of $x_i\mapsto h_2(x)$ in Assumption \ref{assump:h2}(i),  or that of $x_i\mapsto \hat h(x)$ in Remark \ref{remark-2}, or that of  $x_i\mapsto F(x)$ in Remark \ref{remark_1} are upper bounded over the set that contains $\{x^k\}_{k\geq 0}$ then this assumption is satisfied. Note that $x_i^k \in \dom (f_i)$.}. It holds that if $(x^*,y^*,\omega^*)$ is a limit point  of the generated sequence of iADMMn then $(x^*,y^*,\omega^*)$ is an $\frac{1-\tau_1}{\tau_2 \beta}  \|\omega^*\|$-approximate stationary point of Problem  \eqref{CP-main}. When $\tau_1=1$, we have $(x^*,y^*)$ is a stationary point of Problem \eqref{CP-main} (or equivalently, $(x^*,y^*,\omega^*)$ is a critical point of $\mL_\beta$).
  \end{theorem}
  \begin{proof}
      We remind that the update rule of $x_i^k$ in \eqref{eq:xi_update_nosmooth}, \eqref{eq:xi_update} and  \eqref{eq:xi_update_2} are just special cases of \eqref{eq:update_xi_general}. In \eqref{eq:update_xi_general},  $u_i(x_i,z)$ is a block $x_i$ surrogate function of $F(x) + \frac{\beta}{2} \|h(x)\|^2$ and the formulas of $\mG$ in \eqref{mG1}, \eqref{mG2}, and \eqref{mG3}, which respectively correspond to the update in \eqref{eq:xi_update_nosmooth}, \eqref{eq:xi_update}, and \eqref{eq:xi_update_2}, satisfy the condition $\|\mG_i^k(x_i^k,x_i^{k-1}) \| \leq a_i^k \|\Delta x_i^k\|$. 
      
      From \eqref{eq:update_xi_general}, we know that the following inequality holds for all $x_i$
      \begin{equation}
     \label{eq:c1}
     \begin{split}
 &  f_i(x_i^{k+1})  +  \langle h(x^{k+1}_i,x^{k,i-1}_{\ne i}),  \omega^k + \beta\mB y^k\rangle + u_i(x_i^{k+1},x^{k,i-1})  \\
 &\leq f_i(x_i)  +  \langle h(x_i,x^{k,i-1}_{\ne i}),  \omega^k + \beta\mB y^k\rangle + u_i(x_i,x^{k,i-1})\\
 &\qquad \qquad \qquad-\langle \mG_i^k(x_i^k,x_i^{k-1}),x_i-x_i^{k+1}\rangle. 
\end{split}
 \end{equation}
Let $\{(x^{k_n},y^{k_n},\omega^{k_n})\}$ be a subsequence that converges to $(x^*,y^*,\omega^*)$. As $\Delta x^k_i\to 0$ (see Proposition \ref{prop:delta_converge}), we have $x_i^{k_n+1}$, $x_i^{k_n-1}$ and $x_i^{k_n-2}$ also converge to $x_i^*$ for all $i\in [s]$. Consequently, $\|\mG_i^k(x_i^{k_n-1},x_i^{k_n-2})\|\to 0$.  Taking $x_i=x_i^*$ and $k=k_n-1$ in \eqref{eq:c1}, we obtain 
$ 
\limsup_{n\to \infty} f_i(x_i^{k_n}) \leq f_i(x_i^*),
$
which implies $f_i(x_i^{k_n})\to f_i(x_i^*)$ as $f_i$ is lower semi-continuous. Taking $k=k_n-1$ in \eqref{eq:c1} and let $k_n\to \infty$, we obtain the following inequality for all $x_i$
 \begin{equation}
\label{temp2}
\begin{split}
 &  f_i(x_i^*)  +  \langle h(x^*),  \omega^* + \beta\mB y^*\rangle + u_i(x_i^{*},x^{*})  \\
 &\leq f_i(x_i)  +  \langle h(x_i,x^{*}_{\ne i}),  \omega^* + \beta\mB y^*\rangle + u_i(x_i,x^{*}).
 \end{split}
\end{equation}
Note that $u_i(x_i^*,x^*)=F(x^*)+\frac{\beta}{2}\|h(x^*)\|^2$. Hence, \eqref{temp2} implies that for all $x_i$ we have
\begin{equation}
\label{temp30}
\begin{split}
   \varphi(x^*,y^*,\omega^*) + f_i(x_i^*)& \leq    \varphi(x_i,x^*_{\ne i},y^*,\omega^*) + f_i(x_i) \\
   &\qquad+ u_i(x_i,x^*)-F(x_i,x^*_{\ne i})-\frac{\beta}{2}\|h(x_i,x^*_{\ne i})\|^2.
    \end{split}
\end{equation}
Let $ e_i(x_i,z)=u_i(x_i,z)-F(x_i,z_{\ne i})-\frac{\beta}{2}\|h(x_i,z_{\ne i})\|^2$. We have $e_i(x_i^*,x^*)=0$. It follows from \eqref{temp30} that $x_i^*$ is a solution of 
$$\min_{x_i} \varphi(x_i,x^*_{\ne i},y^*,\omega^*) + f_i(x_i) +  e_i(x_i,x^*).
$$
Writing the optimality condition for this problem and noting that $\nabla_{x_i}e_i(x^*_i,x^*)=0$ for all the three surrogates in \eqref{ui1}, \eqref{ui2}, and \eqref{ui3} that respectively corresponds to the update rule of $x_i$ in \eqref{eq:xi_update_nosmooth}, \eqref{eq:xi_update} and  \eqref{eq:xi_update_2}, we obtain 
$$0\in\partial_{x_i} \revise{\mL_\beta}(x^*,y^*,\omega^*)= \partial_{x_i} (\varphi (x^*,y^*,\omega^*) + f_i(x_i^*)).$$ 
Since the essence of the update rule of $y$ is to minimize a block $y$ surrogate function of $\varphi$ (see \eqref{eq:y_update}), we can use the same reasoning for $x_i$ to prove that $0\in \nabla_y \revise{\mL_\beta}(x^*,y^*,\omega^*)$. 

From \eqref{eq:omega_update} we have $h(x^{k_n})+\mB y^{k_n} = \frac{1}{\tau_2 \beta}(\Delta \omega^{k_n} + (1-\tau_1) \omega^{{k_n}-1})$. Furthermore, $\Delta \omega^{k_n}\to 0$ (see  Proposition \ref{prop:delta_converge}), which implies $\omega^{{k_n}-1}\to \omega^*$. 
Hence, we have $$h(x^*) + \mB y^*=\frac{1-\tau_1}{\tau_2 \beta} \omega^*.$$ 
The result follows then. 
  \end{proof}

\begin{theorem}[Global convergence]
\label{global_conver}
 Suppose the conditions for Theorem \ref{theorem-sub} are satisfied and we also assume that $\Theta(x,y)$ has the K{\L} property, the generated sequence $\{(x^k,y^k,\omega^k)\}$ is bounded, and let $\tau_1=\tau_2=1$. Furthermore, together with the existence of the constants $\tilde \eta_i$ in Theorem \ref{theorem-sub} we assume there exist $\bar \eta_i>0$ such that $\eta_i^k \leq \bar \eta_i$  for all $i\in [s]$ and $k\geq 0$, and the constant $B_1$ in \eqref{parameter} satisfies $B_1<\min\{\tilde \eta_i /\bar \eta_i\}$. 
   Then the whole sequence $\{ (x^k,y^k,\omega^k)\}$ converges to a critical point of $\mL_\beta$. 
\end{theorem}
Before proving Theorem \ref{global_conver}, let us have some discussion on the assumptions used for Theorem \ref{global_conver}.
\begin{itemize} 
\item  If $h(x)=\mathcal A x=\sum_{i=1}^s \mathcal A_i x_i$, where $\mathcal A_i$ are linear mappings, and we use the surrogate in \eqref{ui1} then  $l_i^k=\|\revise{\mathcal A_i^\top} \mathcal A_i\|$. In this case, we simply choose $\kappa_i^k$ in the update \eqref{eq:xi_update_nosmooth} to be any constant $\kappa_i \geq \|\revise{\mathcal A_i^\top} \mathcal A_i\|$, then $\eta_i^k$ in Proposition \ref{prop:nsdp_xi} are constants. Hence, $\tilde \eta_i/\bar\eta_i=1$. 
\item It is important to note that we require $\tau_1=\tau_2=1$ for the global convergence, but this condition is not required for Theorem \ref{theorem-sub}. 
 
\item The boundedness assumption of $\{(x^k,y^k,\omega^k)\}$ is necessary for the upcoming proof. Totally similarly to \cite[Proposition 7]{mADMM2022}, we can prove that if $\tau_1=1$, $ran\, h(x)  \, \subseteq Im(\mB) $, $\lambda_{\min} (\mB^{\revise{\top}} \mB)>0$, and $\Theta(x,y)$ is coercive over the feasible set $\{(x,y): h(x) + \mB y=0\}$ then the generated sequence $  \{(x^k,y^k,\omega^k)\}$ is bounded. We omit the proof of this property. 
\item Once the global convergence of iADMMn is guaranteed, by using the same technique as in~\cite[Theorem~2]{Attouch2009} (see \cite[Theorem~2.9]{Xu2013} and  \cite[Theorem~3]{Hien_ICML2020} for some examples of using this technique to establish the convergence rate), we can establish a convergence rate for iADMMn.  The type of convergence rate depends on the value of the K{\L}  exponent.  
\revise{More specifically, when the K{\L}  exponent is 0, the algorithm converges after a finite number of steps,
when the K{\L}  exponent is in $(0, 1/2]$, the algorithm has linear convergence, and when the K{\L}  exponent is in $(1/2, 1)$, the
algorithm has sublinear convergence. The estimation of the K{\L} exponent falls beyond the scope of this paper. 
}
\end{itemize}
We now prove Theorem \ref{global_conver}. 
\revise{  As discussed  at the beginning of Section~\ref{sec:analysis}, we use the same methodology employed in \cite{Bolte2018_ADMM,Hien_iADMM,Titan2020,Hien_BMME2022,Ochs2014,Wang2019,Xu2013,Yashitini2022} to prove the global convergence of iADMMn, which comprises of three main steps (i) derive sufficient decrease property of a Lyapunov function, (ii) derive a subgradient lower bound for the iterates gap, and (iii) using the K{\L} property. Although the general methodology is the same in these papers as its essence is originally from the  general proof recipe of \cite{Attouch2013,Bolte2014}, how to choose a Lyapunov function and how to establish the two properties (i) the sufficient decrease and (ii) the boundedness of subgradient are different in these papers since these steps highly rely on the structure of the problem. In the following, we will prove these two properties. The remaining step to obtain the global convergence are totally similar to the proof of \cite[Theorem 2]{Hien_iADMM}; hence, we omit the details. }
      \begin{proof}
      Denote $z=(x,y,\omega)$. We use the following Lyapunov function 
        \begin{equation}
            \label{lyapunov}
            \begin{split}
            \bar \mL(z,\tilde z)&=\mL_\beta(x,y,\omega) + \sum_{i=1}^s \frac{\tilde\eta_i+B_1 \bar\eta_i}{2}\|x_i-\tilde x_i\|^2\\
            &\qquad\qquad+ \frac{(1+B_2)C_3}{2}\|y-\tilde y\|^2 + C_1\|\mB^{\revise{\top}} (\omega - \tilde \omega)\|^2. 
            \end{split}
        \end{equation}
        
        \paragraph{Sufficient decrease property} From Proposition \ref{prop:Lk-recursive} and the \revise{condition} $\tilde \eta_i \leq \eta_i^k$, $\gamma_i^k\leq B_1 \eta_i^{k-1}\leq B_1 \bar\eta_i$, and $8C_2 L_G^2\leq B_2 C_3$,  we obtain 
       \[ 
       \begin{split}
      & \bar\mL(z^{k+1},z^{k}) +\sum_{i=1}^s \frac{\tilde\eta_i-B_1 \bar\eta_i}{2}\big(\|\Delta x_i^{k+1}\|^2 + \|\Delta x_i^k\|^2\big)  \\
       &\qquad\qquad+ \frac{(1-B_2)C_3}{2} \big(\|\Delta y^{k+1}\|^2 + \|\Delta y^k\|^2\big)\leq \bar\mL(z^{k},z^{k-1}).
       \end{split}
       \]
        \paragraph{Boundedness of subgradient} In the following, we will work on the bounded set that contains the generated sequence (as we assume the generated sequence is bounded). 
        The optimality condition of \eqref{eq:update_xi_general} gives us 
        \begin{equation}
       \label{temp3}
        \begin{split}
              &\mG_i^k(x_i^k,x_i^{k-1}) -\nabla_{x_i} h(x_i^{k+1},x^{k,i-1}_{\ne i})^\top( \omega^k+\beta \mB y^k)\\
              &\qquad\in \revise{ \partial_{x_i} f_i(x_i^{k+1}) + \nabla_{x_i} u_i(x^{k+1}_i,x^{k,i-1}).}
        \end{split}
        \end{equation}
       \revise{Note that $ \nabla_{x_i} u_i(x_i^{k+1},x^{k+1})=\nabla_{x_i}\big(F(x^{k+1})+ \beta h_2(x^{k+1})\big)$, $\nabla_{x_i} u_i(\cdot,\cdot)$ is continuously differentiable, and we are working on the bounded set containing the generated sequence. So there exists a constant $\hat L_i$ such that 
       \begin{equation}
       \label{temp-revise}\|\nabla_{x_i} u_i(x_i^{k+1},x^{k,i-1}) - \nabla_{x_i}\big(F(x^{k+1})+ \beta h_2(x^{k+1})\big)\| \leq \hat  L_i \|x^{k+1}-x^{k,i-1}\|.
       \end{equation} } 
        From \eqref{temp3} we know that there exists $d_i^{k+1} \in \partial_{x_i} f_i(x_i^{k+1})$ 
        such that 
        \begin{equation} 
        \label{temp4}
        \mG_i^k(x_i^k,x_i^{k-1}) -\nabla_{x_i} h(x_i^{k+1},x^{k,i-1}_{\ne i})^\top ( \omega^k+\beta \mB y^k) = d_i^{k+1} + \revise{\nabla_{x_i} u_i(x^{k+1}_i,x^{k,i-1})}.
        \end{equation}
        Since $h_2$ is continuously differentiable and 
        \begin{equation*}
        \begin{split}
            \partial_{x_i} \mL_\beta(z^{k+1})&=\partial_{x_i}f_i(x_i^{k+1}) + \nabla_{x_i}\big(F(x^{k+1})+\beta h_2(x^{k+1})\big) \\
            & \qquad\qquad+ \nabla_{x_i} h(x^{k+1})^\top (\omega^{k+1} + \beta \mB y^{k+1}),
            \end{split}
        \end{equation*}
        we have \revise{(denote $\bar \xi_i^{k+1}= \nabla_{x_i}\big(F(x^{k+1})+ \beta h_2(x^{k+1})\big)$ \[
       D_i^{k+1}:=d_i^{k+1} + \bar \xi_i^{k+1} + \nabla_{x_i} h(x^{k+1})^\top (\omega^{k+1} + \beta \mB y^{k+1}) \in \partial_{x_i}\mL_\beta(z^{k+1}).\]} 
        Together with \eqref{temp4} and \eqref{temp-revise} we get
        \begin{equation}
        \label{temp5}
        \begin{split}
            \|D_i^{k+1}\|&=\|d_i^{k+1} + \xi_i^{k+1} +  \bar\xi_i^{k+1}-\xi_i^{k+1} + \nabla_{x_i} h(x^{k+1})^\top (\omega^{k+1} + \beta \mB y^{k+1})\|\\
            &\leq \big\|\mG_i^k(x_i^k,x_i^{k-1}) -\nabla_{x_i} h(x_i^{k+1},x^{k,i-1}_{\ne i})^\top ( \omega^k+\beta \mB y^k) \\
            &\quad+ \nabla_{x_i} h(x^{k+1})^\top (\omega^{k+1} + \beta \mB y^{k+1}) \big\| + \| \bar\xi_i^{k+1}- \revise{\nabla_{x_i} u_i(x^{k+1}_i,x^{k,i-1})} \|\\
            &\leq \|\mG_i^k(x_i^k,x_i^{k-1})\|+ \|\nabla_{x_i} h(x^{k+1})^\top \| (\|\Delta \omega^{k+1}\|+\beta \|\mB \Delta y^{k+1}\|) \\
            &\qquad\qquad+\hat L_i \|x^{k+1}-x^{k,i-1}\|.
        \end{split}
        \end{equation}
  On the other hand, as  $h_i$, $i=1,\ldots,q$, are continuously differentiable, hence $\|\nabla_{x_i} h(x^{k+1})^\top\|$ is bounded on the bounded set that contains the generated sequence. Since $\tau_1=\tau_2=1$, we get from \eqref{eq:gammaB_4}   that 
  \begin{equation} 
 \label{temp7}
 \|\Delta \omega^{k+1}\|^2 \leq \frac{1}{\sigma_{\mB} }\|\revise{\mB^\top}\Delta \omega^{k+1} \|^2 \leq \frac{1}{\sigma_{\mB} }(2L_G^2\|\Delta y^{k+1}\|^2 + 8 L_G^2 \|\Delta y^{k}\|^2). 
 \end{equation}    
  Therefore, from \eqref{temp5} we derive that 
 \begin{equation} 
 \label{temp6} 
 \|D_i^{k+1}\|\leq a_1( \|\Delta x^{k+1}\| +  \|\Delta x^{k}\|+\|\Delta y^{k+1}\|+ \|\Delta y^{k}\|)
 \end{equation} for some positive constant $a_1$. 
  Note that 
  $$D_y^{k+1}:=\nabla G(y^{k+1}) + \mB^{\revise{\top}} (\omega^{k+1} + \beta( h(x^{k+1})+\mB y^{k+1})) \in \nabla_y \mL_\beta (z^{k+1}),
  $$ 
  which together with\eqref{eq:nabM} and \eqref{temp7} leads to 
 \begin{equation}
 \label{temp8}
     \begin{split}
      \|D_y^{k+1}\| &= \|\nabla G(y^{k+1})- \nabla G(y^{k})+  \mB^{\revise{\top}}  \Delta \omega^{k+1}-L_G \Delta y^{k+1}\|   \\
      &\leq 2L_G \| \Delta y^{k+1}\| + \|\mB^{\revise{\top}} \Delta \omega^{k+1}\|  \\
      &\leq  a_2( \| \Delta y^{k+1}\|+\| \Delta y^{k}\|)
     \end{split}
 \end{equation} 
 for some positive constant $a_2$. On the other hand, we have
 \begin{equation}
     \label{temp9}
 \|\nabla_\omega  \mL_\beta (z^{k+1}) \| = \| h(x^{k+1})+\mB y^{k+1} \|=1/\beta\|\Delta \omega^{k+1}\|, \end{equation}
  and
        \begin{equation} 
        \label{partial_Lbar}
        \begin{split} 
        \partial \bar\mL(z, \tilde{z})=\partial \mL_\beta( z)+ &\partial \Big( \sum_{i=1}^s \frac{\tilde\eta_i + B_1 \bar\eta_i}{2}\|x_i - \tilde x_i \|^2+ B_2C_3 \|y-\tilde y\|^2 \\
        &\qquad + C_1\|\mB^{\revise{\top}} (\omega - \tilde \omega)\|^2 \Big).
        \end{split}
        \end{equation}
        Therefore, from \eqref{partial_Lbar}, \eqref{temp9}, \eqref{temp7},  \eqref{temp8}, and \eqref{temp5}, it is not difficult to derive that 
$$\|D^{k+1}\| \leq a_3( \|\Delta x^{k+1}\| +  \|\Delta x^{k}\|+\|\Delta y^{k+1}\|+ \|\Delta y^{k}\|), $$
for some positive constant $a_3$ and $D^{k+1}\in \partial \bar\mL(z^{k+1},z^k)$. 
    \end{proof}

\section{Numerical results} \revise{In this section, we test iADMMn on Problem \eqref{log_MF}. All tests are performed  using Matlab
R2021b on a laptop 2.5 GHz Intel Core i5
of 16GB RAM. The code is available from \url{https://github.com/LeThiKhanhHien/iADMMn}. 

Problem~\eqref{log_MF} can be rewritten in the form of \eqref{CP-main} as follows.
\begin{equation}
    \label{log_MF_2}
    \begin{split}
       \min_{U\in\mbbR^{m\times r},V\in\mbbR^{r\times n}} &  \quad F(U,V) + G(W)\\
      s.t\quad &  UV-W=0,
    \end{split}
\end{equation}
where $F(U,V)=\frac{\lambda_d}{2} \|U\|_F^2 + \frac{\lambda_t}{2}\|V\|_F^2$ and $G(W)=\sum_{i,j} (1+c y_{ij}-y_{ij})  \log (1+\exp(W_{ij}))-c y_{ij} W_{ij}$. The augmented Lagrangian of Problem~\eqref{log_MF_2} is 
$$
\mL_\beta(U,V,W,\omega)= F(U,V) + G(W) + \langle UV-W,\omega\rangle + \frac{\beta}{2}\|UV-W\|^2.$$
The update in \eqref{eq:xi_update_nosmooth} for $U$ (we keep $F$ as a surrogate of itself) is 
\[\begin{split}
U^{k+1}&\in\arg\min_{U}\Big\{\frac{\lambda_{d}}{2}\|U\|_{F}^{2}+\langle U,(w-\beta W)V^{\top}\rangle+\beta\langle U_{ex}VV^{\top},U\rangle\\
&\qquad\qquad\qquad+\frac{\beta\|VV^{\top}\|}{2}\|U-U_{ex}\|^{2}\Big\}\\
&=\frac{\beta\|VV^{\top}\|}{\beta\|VV^{\top}\|+\lambda_{d}}U_{ex}-\frac{1}{\beta\|VV^{\top}\|+\lambda_{d}}wV^{\top}-\frac{\beta(U_{ex}V-W)V^{\top}}{\beta\|VV^{\top}\|+\lambda_{d}},
\end{split}
\]
where $W=W^k, V=V^k, \omega=\omega^k$ and $U_{ex}=U^k + \alpha_U^k (U^k-U^{k-1})$. Similarly, the update of $V$ is 
$$V^{k+1}=\frac{\beta\|U^{\top}U\|}{\beta\|U^{\top}U\|+\lambda_{t}}V_{ex}-\frac{1}{\beta\|U^{\top}U\|+\lambda_{t}}U^{\top}w-\frac{\beta U^{\top}(UV_{ex}-W)}{\beta\|U^{\top}U\|+\lambda_{t}},
$$
where $U=U^{k+1}, W=W^k, \omega=\omega^k$, and $V_{ex}=V^k + \alpha_V^k (V^k-V^{k-1})$. The update of $W$ in \eqref{eq:y_update} is 
$$
W^{k+1}=\frac{1}{\beta+L_{G}}\left(L_{G}W^k-\nabla 
 G(W^k)+\omega^k+\beta U^{k+1}V^{k+1}\right),$$ 
and the update of $\omega$ is 
$$\omega^{k+1} = \tau_1 \omega^k + \tau_2 \beta (U^{k+1}V^{k+1}-W^{k+1}).
$$
We choose the following extrapolation parameters satisfying $\gamma_i^k \leq B_1 \revise{\eta_i^{k-1}}$:
\begin{equation*}
\begin{array}{ll}
B_1=0.9999, t_0=1, t_k=\frac12(1+\sqrt{1+4t_{k-1}^2}), \\
\alpha_U^k =\min\Big\{\frac{t_{k-1}-1}{t_k},B_1\sqrt{ \frac{\|V^{k-1}(V^{k-1})^\top\|}{ \|V^{k}(V^{k})^\top\|}}\Big\}, \\
\alpha_V^k =\min\Big\{\frac{t_{k-1}-1}{t_k},B_1\sqrt{ \frac{\|(U^{k-1})^\top U^{k-1}\|}{ \|(U^{k})^\top U^{k}\|}}\Big\}.
\end{array}
\end{equation*}

To generate a sparse data $Y\in\mbbR^{m\times n}$ with each element $y_{ij}\in \{0,1\}$, we use  Matlab command \texttt{sprand(m,n,s);} and assign \texttt{Y(Y>0)=1;} In the upcoming experiments, we set $s=0.1$ (that is, 90\% of the elements of $Y$ are 0), $r=100$, $c=1$, $\lambda_d=\lambda_t=\frac{1}{4}$, and $\beta=1$. For each size $(m,n)\in\{(200,200), (200,1000), (1000,200)\}$, we generate 5 random  sparse matrices $Y$. And for each $Y$, we generate 5 random initial points. We run each algorithm with the same initial point and the same running time: 10 seconds for the 
size $(m,n)=(200,200)$ and 25 seconds for the size $(m,n)\in \{(200,1000), (1000,200)\}$.  

We compare iADMMn with GD - the alternating gradient descent method, which alternatively updates $U$ and $V$ by gradient descent step, and implement 
three iADMMn versions corresponding to $(\tau_1,\tau_2)\in \{(0.1,0.1), (0.5,0.5), (1,1)\}$, together with their non-inertial versions. 
We compute mean of the objective values of Problem~\eqref{log_MF} over 25 trials (5 random datasets and 5 random initial points) and report the evolution of their log with respect to time in Figure~\ref{fig:mean}. We also report the mean $\pm$ std of the final objective values in Table \ref{tab:report}. 

\begin{figure}
    \centering
    \begin{tabular}{cc}
      (m,n)=(200,200) & (m,n)=(200,1000)
     \\ 
    \includegraphics[scale=0.35]{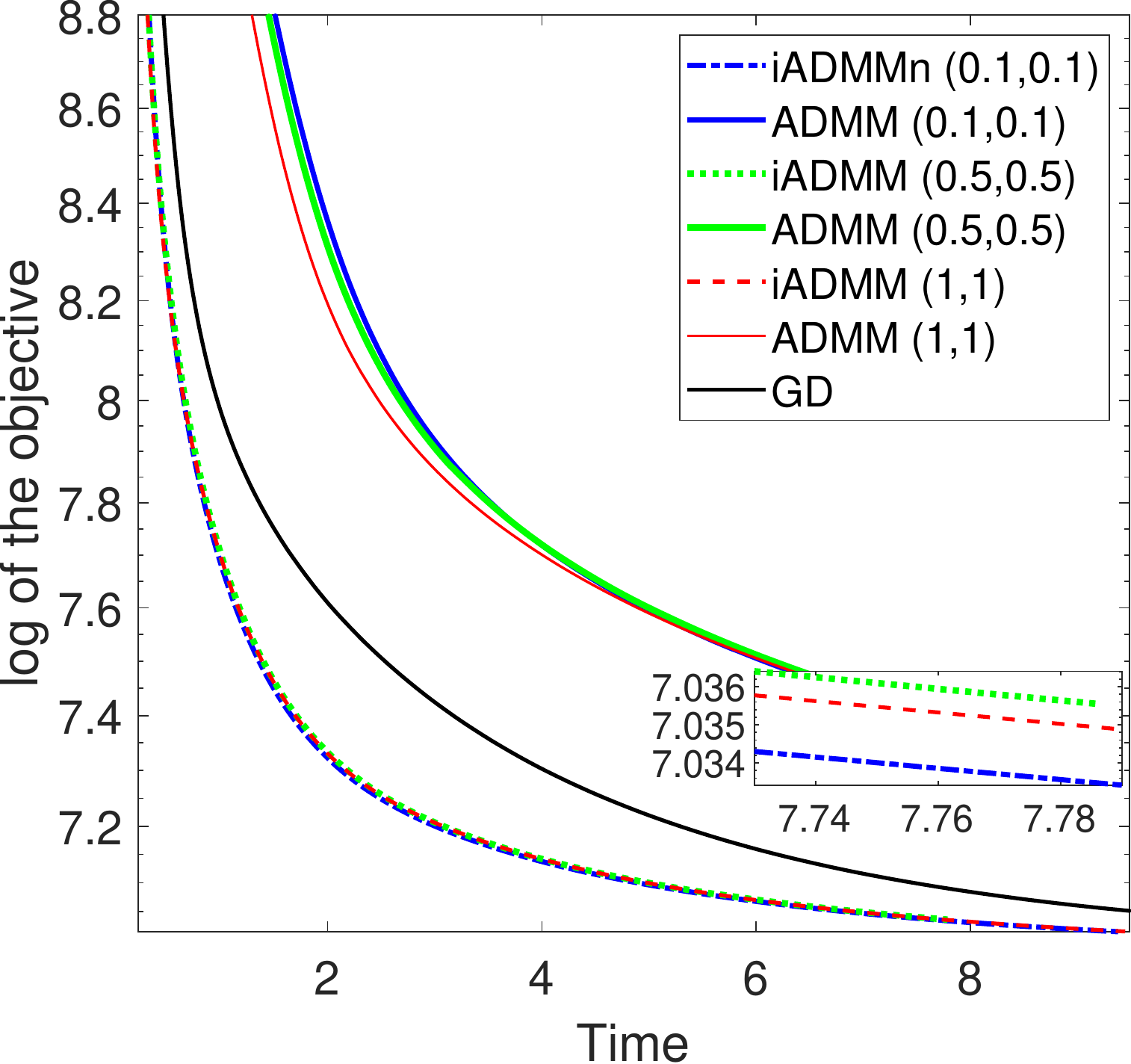}  &\includegraphics[scale=0.35]{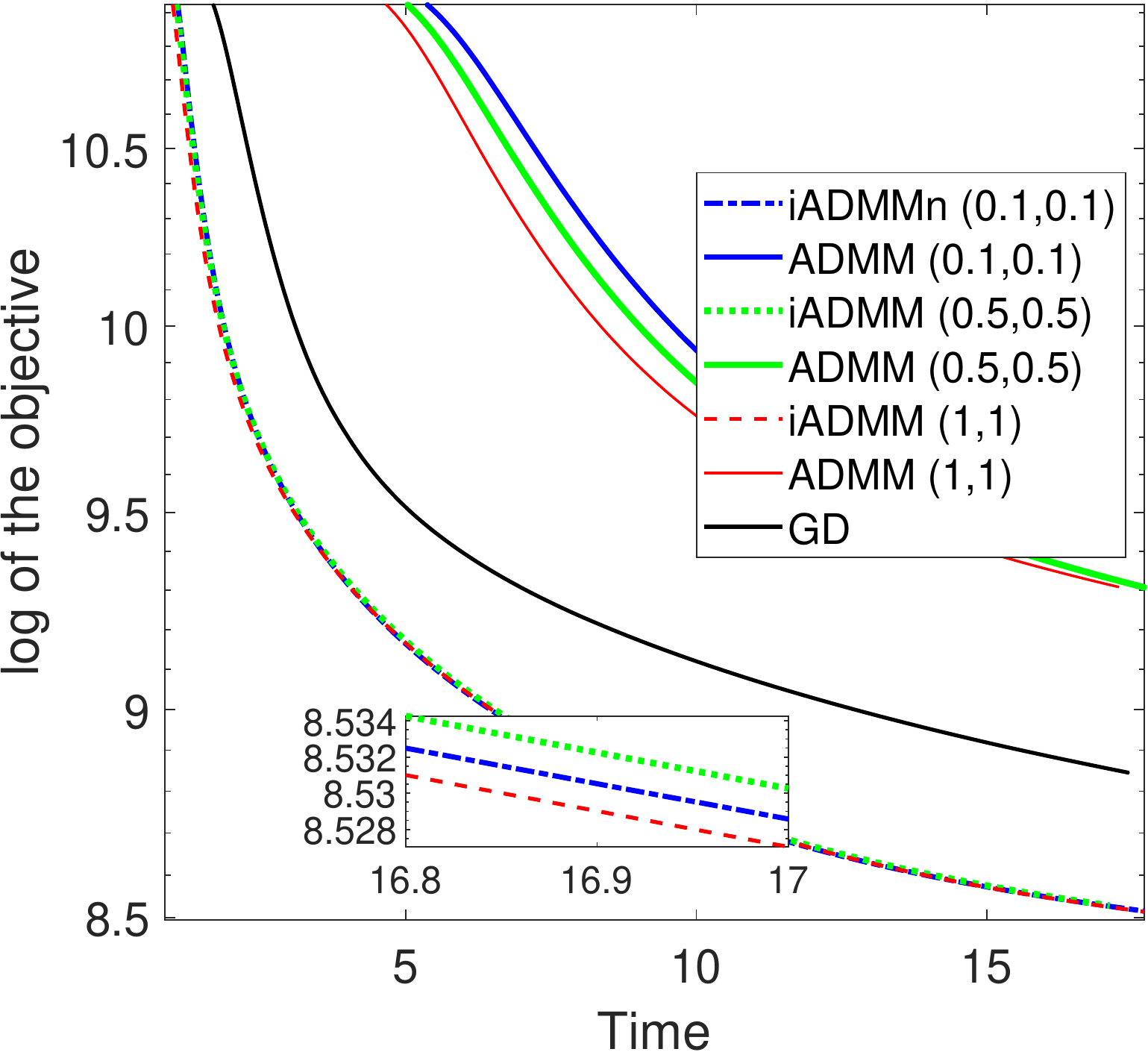}\\
    \multicolumn{2}{c}{(m,n)=(1000,200) } \\
         \multicolumn{2}{c}{\includegraphics[scale=0.35]{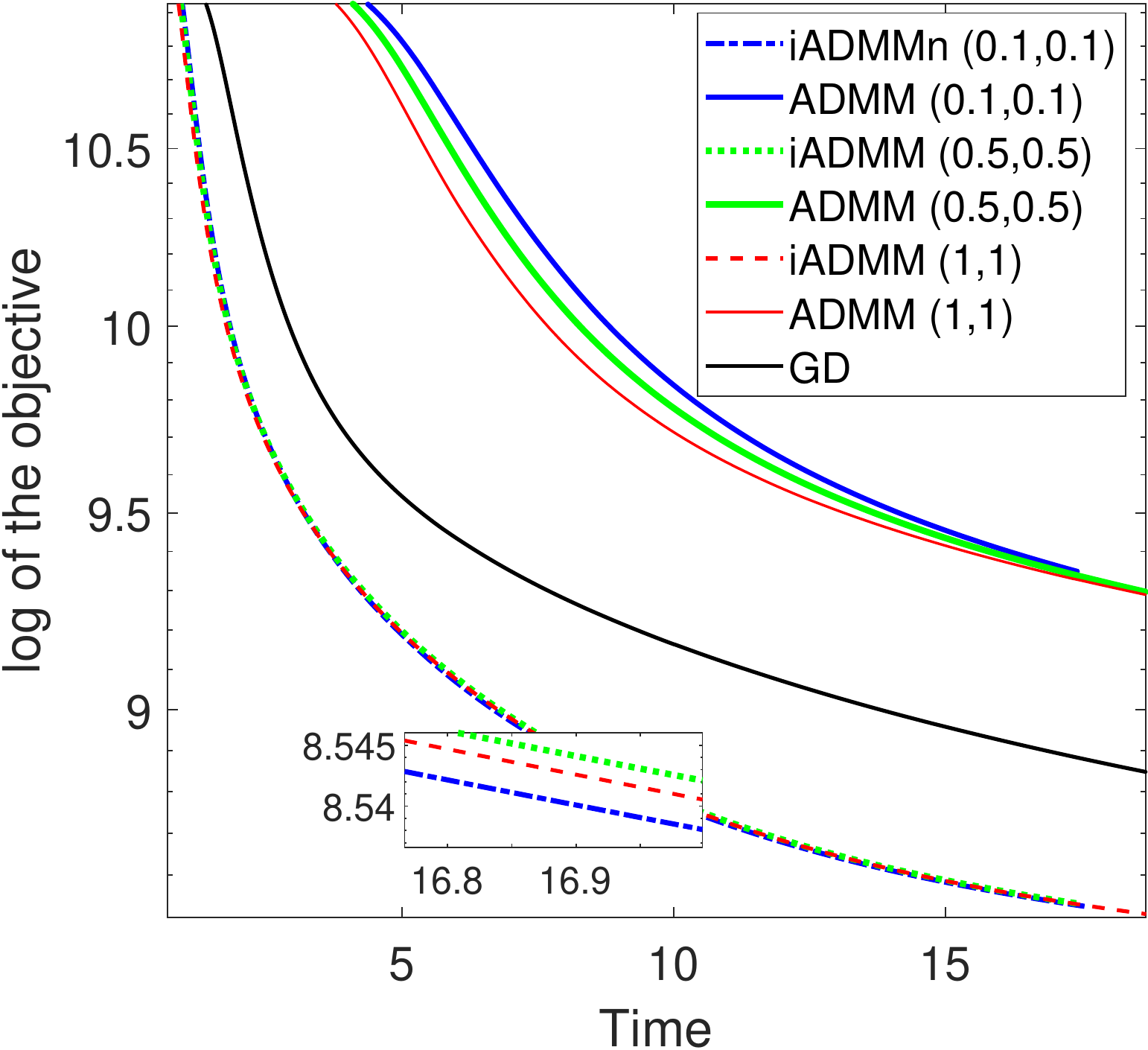}}
       
     \end{tabular}
    \caption{Evolution of the log of mean of the objective values with respect to time. }
    \label{fig:mean}
\end{figure}

\begin{table}\captionsetup{font=small}
\caption{Mean  (and std ) of the  final objective values of Problem~\eqref{log_MF}  over 25 trials. The best means are highlighted in bold. } \vspace{0.08in}
    \label{tab:report}
\centering
      \adjustbox{scale=0.9}{   \begin{tabular}{cccc}
    \toprule
    Algorithm \Big/(m,n)& (200x200) & (200,1000) & (1000,200) \\ \midrule  
    iADMMn (0.1,0.1)&$\mathbf{ 1.106\times 10^{3}}$ (2.977)& $4.807 \times 10^3$ (50.639) & $\mathbf{4.846\times 10^3}$ (68.128)\\\midrule  
    ADMMn (0.1,0.1) & 1.469 $\times 10^{3}$ (15.751) & $1.035\times 10^4$ (427.431)&$1.052\times 10^4$ ( 271.288)\\\midrule  
    iADMMn (0.5,0.5) & 1.108 $\times 10^{3}$ (5.994) &  $4.813\times 10^3$ (56.740)&$4.864\times 10^3$ (59.246)\\
              \midrule    
             ADMMn (0.5,0.5)& 1.477 $\times 10^{3}$ (28.267)& $1.020 \times 10^4$ (178.253) & $1.047\times 10^4$ (124.764)\\
               \midrule    
              iADMMn (1,1)& $1.107\times 10^{3}$ (2.430) & $ \mathbf{4.800 \times 10^3}$ ( 24.680) & $4.855\times 10^3$ (30.969)\\
               \midrule    
              ADMMn (1,1)& $1.476\times 10^{3}$ ( 11.517) & $ 1.013 \times 10^4$ ( 204.938)& $1.042\times 10^4$ (122.554)\\ \midrule
              GD & $1.144\times 10^{3}$ (3.476) & $6.491 \times 10^3$ (102.308)& $6.725 \times 10^3$ (44.983)\\
              \bottomrule
    \end{tabular}}%
       \end{table}
We observe from Figure~\ref{fig:mean} and Table \ref{tab:report} that iADMMn consistently outperforms its non-inertial version and the alternating gradient descent method. iADMMn with $\tau_1=\tau_2=0.1$ more often produces better final objective values than the other two iADMMn variants  with $\tau_1=\tau_2=0.5$ and $\tau_1=\tau_2=1$. 
}
\section{Conclusion}
\label{sec:conclusion}
We have analyzed iADMMn, an inertial  alternating direction method of multipliers for solving Problem \eqref{CP-main}.  In essence, iADMMn is an extension of iADMM proposed in \cite{Hien_iADMM}:  iADMMn allows the nonlinearity of the coupling constraint whereas iADMM only considers linear coupling constraint.  iADMMn is also an extension of \cite[Algorithm 2]{mADMM2022}:  iADMMn allows inertial term in the updating of $x_i$ whereas \cite[Algorithm 2]{mADMM2022} does not.  On the other hand, iADMMn considers a non-standard update rule for the multiplier $\omega$ when it allows a scaling factor $\tau_1\in (0,1]$ for $\omega^k$ whereas both iADMM and \cite[Algorithm 2]{mADMM2022} use standard rule $\tau_1=1$. Theorem \ref{theorem-sub} shows that every limit point of the generated sequence is an $\varepsilon$-stationary point of Problem \eqref{CP-main}. In Theorem \ref{global_conver}, $\tau_1=1$ is required  to guarantee a global convergence for the generated sequence. However, when $\tau_1\ne 1$, if there exist many limit points of the generated sequence or not is still an open question for us. We take it as a future research direction.  

\revise{\section*{Acknowledgement}
We express our sincere appreciation to the reviewers for their comments, which greatly helped improve the paper.}

\section*{Data availability Statement}
The manuscript has \revise{synthetic data}. 
\bibliography{sdADMM}
\bibliographystyle{spmpsci} 
\end{document}